\documentclass[11pt,leqno]{article}
\usepackage{amsthm,amsfonts,amssymb,amsmath,oldgerm}
\usepackage{epsfig}
\usepackage{color, hyperref}
\numberwithin{equation}{section}
\usepackage[thinlines]{easybmat}

\def\eps{\varepsilon }

\newcommand\R{\mathbb R}
\newcommand\C{\mathbb C}
\newcommand\N{\mathbb N}

\def\eps{\varepsilon}

\newcommand\br{\begin{remark}}
\newcommand\er{\end{remark}}
\newcommand\bp{\begin{pmatrix}}
\newcommand\ep{\end{pmatrix}}
\newcommand\be{\begin{equation}}
\newcommand\ee{\end{equation}}

\newcommand{\bs}{\left\{}
\newcommand\es{\right.}
\newcommand\ba{\begin{array}}
\newcommand\ea{\end{array}}

\newcommand{\bap}{\begin{app}}
\newcommand{\eap}{\end{app}}
\newcommand{\begs}{\begin{exams}}
\newcommand{\eegs}{\end{exams}}
\newcommand{\beg}{\begin{example}}
\newcommand{\eeg}{\end{exaplem}}
\newcommand{\bpr}{\begin{proposition}}
\newcommand{\epr}{\end{proposition}}
\newcommand{\bt}{\begin{theorem}}
\newcommand{\et}{\end{theorem}}
\newcommand{\bc}{\begin{corollary}}
\newcommand{\ec}{\end{corollary}}
\newcommand{\bl}{\begin{lemma}}
\newcommand{\el}{\end{lemma}}
\newcommand{\bd}{\begin{definition}}
\newcommand{\ed}{\end{definition}}
\newcommand{\brs}{\begin{remarks}}
\newcommand{\ers}{\end{remarks}}

\newtheorem{theo}{Theorem}[section]

\newtheorem{cor}[theo]{Corollary}

\newtheorem{exams}[theo]{Examples}

\numberwithin{equation}{section}

\newcommand{\CC}{{\mathbb C}}

\newtheorem{theorem}{Theorem}[section]
\newtheorem{proposition}[theorem]{Proposition}
\newtheorem{corollary}[theorem]{Corollary}
\newtheorem{lemma}[theorem]{Lemma}
\newtheorem{definition}[theorem]{Definition}

\newtheorem{example}[theorem]{Example}
\newtheorem{remark}[theorem]{Remark}

\newcommand{\beq}{\begin{equation}}
\newcommand{\eeq}{\end{equation}}

\newtheorem{remarks}[subsection]{Remarks}
\newcommand{\pd}{\partial}
\renewcommand{\hat}{\widehat}

\title{The reflection coefficient of a fractional reflector}
\author{Laurent Demanet\thanks{Department of Mathematics, Massachusetts Institute of Technology, 77 Massachusetts Ave, Cambridge, MA 02139, USA} $\,$ and Olivier Lafitte\thanks{LAGA, UMR7539, CNRS and USPN, 99 avenue Jean-Baptiste Cl\'ement, F-93430}\thanks{IRL 3457 CRM-CNRS, Centre de recherches math\'ematiques,  chemin de la Tour, Universit\'e de Montr\'eal, QC, Canada,}}
\begin{document}
\maketitle

\begin{abstract}

This paper considers the question of characterizing the behavior of waves reflected by a fractional singularity of the wave speed profile, i.e., 
of the form 
\[
c(x_1, x_2, x_3) = c_0 \left(1 + \left( \frac{x_1}{\ell}\right)_{+}^\alpha \right)^{-1/2},
\]
for $\alpha > 0$ not necessarily integer. We first focus on the case of one spatial dimension and a harmonic time dependence. 
We define the reflection coefficient $R$ from a limiting absorption principle. We provide
an exact formula for $R$ in terms of the solution to a Volterra equation. We obtain the asymptotic limit of this coefficient
in the large $\ell \omega / c_0$ regime as
\[
R = \frac{\Gamma(\alpha + 1)}{(2 i)^{\alpha + 2}} \left( \frac{c_0}{\ell \omega} \right)^{\alpha} + \mbox{lower order terms.}
\]
The amplitude is proportional to $\omega^{-\alpha}$, and the phase rotation behavior is obtained from the $i^{-(\alpha+2)}$ factor.
The proof method does not rely on representing the solution by special functions, since $\alpha > 0$ is general.

In the multi-dimensional layered case, we obtain a similar result where the nondimensional variable $\ell \omega / c_0$ is modified
to account for the angle of incidence. The asymptotic analysis now requires the waves to be non-glancing. 
The resulting reflection coefficient can now be interpreted as a Fourier multiplier of order $- \alpha$.

In practice, the knowledge of the dependency of both the amplitude and the phase of $R$ on $\omega$ and $\alpha$ might be able
to inform the kind of signal processing needed to characterize the fractional nature of reflectors, for instance in geophysics.

\end{abstract}

\section*{Acknowledgments}

L.D. is supported in part by the Air Force Office of Scientific Research under grant
FA9550-17-1-0316. O. L.  thanks MIT for their hospitality during short visits from 2013 to 2022 and L.D. thanks LAGA-USPN for their hospitality in 2021 and 2022. Partial results were presented at the Days on Diffraction conference in 2014 \cite{DD2014}, and the Waves 2015 conference \cite{W2015}. 


\section{Introduction}

In this paper, we study solutions of the acoustic wave equation in three dimensions,
\be
\label{wave-eq}
\left( \frac{1}{c^2(\mathbf{x})} \pd_{t^2} - \Delta_{\mathbf{x}} \right) u = 0, \qquad \mathbf{x} = (x_1,x_2,x_3) \in \R^3,
\ee
where $c(\mathbf{x})$ presents a singularity of fractional type across the planar interface $x_1 = 0$.

\subsection{Motivation}

Geophysics offers many scenarios of wave reflections, such as point diffractions, reflections off of a sharp planar or curved interface, scattering from rough interfaces, etc. This paper considers another important scenario: that of a ``soft" planar interface with a fractional jump in the value of the wave speed, of the form (constant $+ (x_1)_{+}^{\alpha}$) for some $\alpha > 0$, where the interface is at $x_1 = 0$, and $y_{+}$ is the positive part of $y$. This model might arise as is, or might be the result of upscaling from microscopic mixtures of two materials with linearly varying volume fractions, as predicted by percolation theory applied to random mixture models \cite{BernabeHerrmann2004}. The planar fractional reflector model is a special case of a layered model, commonplace in geophysics. This paper does not address the case of fine layering -- alternating thin layers of different materials or sediments. Furthermore, we only consider acoustic (P) waves in this paper.

\bigskip

The notion of reflection coefficient is standard for acoustic waves propagating in a medium with a discontinuity along a planar interface. For a jump in the wave speed, or in the index of refraction, the Fresnel equations predict the amplitudes of the reflected and transmitted waves as a function of the angle of incidence. The case (constant + $(x_1)_+$) is also known and can be handled with Airy functions \cite{Godin}, which we rehearse in a later section for completeness.

\bigskip

It is much less well known how to deal with other kinds of singularities in the medium properties, such as the square root (constant + $(x_1)^{1/2}_+$) singularity. It was argued on physical grounds that reflection about a singular interface characterized by a fractional exponent should be the result of the action of a fractional integrator, of the same order, on the incident wave. See for instance \cite{Herrmann2001, HerrmannSEG2001, Herrmann2003}. Recognizing the fractional type of a reflector from the particular shape of the oscillations of the recorded waves in a seismic trace is an interesting interpretation question in seismic stratigraphy.

\subsection{Setup}

This paper makes the ``fractional integrator" heuristic precise by characterizing the reflection coefficient as a pseudodifferential  operator of fractional order. The dependence of the amplitudes of the reflected and transmitted waves on the angle of incidence appears explicitly in the expression of the pseudodifferential symbol. 

\bigskip

In the presence of a singularity of $c(x_1, x_2, x_3)$ at $x_1 = 0$, to each incident wave (in $x_1 < 0$) corresponds a transmitted and a reflected wave (both in $x_1 \geq 0$). The reflection coefficient $R$ is the amplitude of the reflected wave, for this unit-normalized incident wave. More precisely, with $\mathbf{x} = (x_1, x_2, x_3)$ and $\mathbf{k} = (k_1, k_2, k_3)$, we consider a fixed incident plane wave
\begin{equation}\label{eq:incident}
u_i(\mathbf{x}, t) = e^{i (\omega t - \mathbf{k} \cdot \mathbf{x})}, \qquad\qquad \omega = | \mathbf{k} | c_0,
\end{equation}
which is clearly a wave that propagates in the direction of $\mathbf{k}$ as $t$ grows, right-going (obliquely) if we further assume $k_1 > 0$. All the waves in this paper have a harmonic time dependence $e^{i \omega t}$. The complete wave $u(\mathbf{x}, t)$ is obtained as a superposition
\begin{equation}\label{eq:RT-intro}
u = u_i + R u_r = T u_t,	\qquad\qquad \mbox{($u_r$ reflected, $u_t$ transmitted)} 
\end{equation}
The concept of transmitted wave $u_t$ is clear in a uniform medium $c_0$: it is a right-going plane wave with $k_1 > 0$, like $u_i$. 

In order to generalize the concept of right-going, or outgoing to $x_1 > 0$, in the case when $c$ depends on $x_1$, observe that a right-going wave like 
\[
e^{- i \mathbf{k} \cdot \mathbf{x}}, \qquad k_1 = \sqrt{\frac{\omega^2}{c_0^2} - k_2^2 - k_3^2} > 0,
\]
corresponds a limiting absorption principle: as $k_1$ is complexified as $k_1 (1 + i \sigma)$, the wave has a $e^{\sigma x_1}$ prefactor, and is now decaying as $x_1 \to + \infty$ provided $\sigma < 0$. This spatial decay can also be seen as a temporal decay (dissipation) since the wave reaches large $x_1 > 0$ for large $t$. In the case when $c$ depends on $x_1$, the limiting absorption principle is formalized by a complex extension at the level of the model equation (\ref{eq:model-sigma}). The definitions of outgoing and incoming waves, respectively at $\pm \infty$, follow and can be found in Definition \ref{def:polarized-waves}. The transmitted wave is properly defined as a member of the one-dimensional space of waves outgoing at $+\infty$, which in turn gives meaning to $R$ in (\ref{eq:RT-intro}).

\subsection{Main result}

In this paper we consider the ``fractional ramp" model for the wave speed,
\be\label{eq:c_lambda}
c^{-2}(\mathbf{x}) = c_0^{-2} \left[ 1 + \left( \frac{x_1}{\ell} \right)_{+}^\alpha \right],
\ee
where $\ell > 0$, $\alpha > 0$, and $( \cdot )_+$ denotes the positive part. The length scale $\ell$ is a skin depth, and $\ell^{- \alpha}$ can be interpreted as the strength of the fractional reflector.

The incident wave in $x_1 < 0$ is taken as (\ref{eq:incident}), where we further write the wave vector as
\[
(k_1, k_2, k_3) = \frac{\omega}{c_0} \, \left( \sqrt{1 - \eta^2} \, , \eta_2, \eta_3 \right), \qquad \eta^2 = \eta_2^2 + \eta_3^2.
\]
The important non-dimensional parameter is
\begin{equation}\label{eq:theta-def}
\theta = \left( \frac{c_0}{\ell \omega} \right)^{\alpha}  (1 - \eta^2)^{- \frac{\alpha + 2}{2}}.
\end{equation}
Our main result is as follows.

\begin{theorem}\label{teo:main}
For all $\alpha > 0$, and as $\theta \to 0$,
\[
R(\theta) = \frac{\Gamma(\alpha+1)}{(2i)^{\alpha+2}} \, \theta  + O \left( \theta^{1 + \min \left( \frac{1}{\alpha}, 1 \right) } \right).
\]
\end{theorem}

The asymptotic regime of small $\theta$ corresponds to the following situation:
\begin{itemize}
\item High frequency ($\omega$ large compared to $c_0 / \ell$); or
\item Weak reflector ($\ell^{-\alpha}$ small); and in both cases
\item Non-grazing incidence ($\eta < 1$).
\end{itemize}

As we explain in the sequel, this result extends to the case when $\left( x_1 / \ell \right)_{+}^\alpha$ is replaced by a function $\lambda \left( \left( x_1 / \ell \right)_{+}^\alpha \right)$ where $\lambda(x)$ is smooth and equal to $x$ in a neighborhood of the origin.

Our construction assumes that the fixed-$\omega$ and fixed-$\eta$ incident wave comes with a unit coefficient, and details the coefficient $R$ of the corresponding reflected wave. If however the incident wave is a superposition of such modes for different $\omega$ and $\eta$ in $x_1 < 0$, such as
\[
u_i(t; x_1, x_2. x_3) = \frac{1}{(2\pi)^3} \int \!\!\! \int e^{i \omega \left( t - \frac{x_1}{c_0} \sqrt{1 - \eta^2} \right) } e^{i \frac{\omega}{c_0} (\eta_2 x_2 + \eta_3 x_3)} \hat{u}_{i}(\omega; 0, \eta_2, \eta_3)  \, d\eta_2 d\eta_3 d\omega,
\]
then the coefficient $R$ acts as a pseudodifferential operator in the expression of $u_r$ in $x_1 < 0$:
\[
u_r(t; x_1, x_2. x_3) = \frac{1}{(2\pi)^3} \int \!\!\! \int e^{i \omega \left( t + \frac{x_1}{c_0} \sqrt{1 - \eta^2} \right) } e^{i \frac{\omega}{c_0} (\eta_2 x_2 + \eta_3 x_3)} R(\theta(\omega, \eta)) \hat{u}_{i}(\omega; 0, \eta_2, \eta_3)  \, d\eta_2 d\eta_3 d\omega.
\]

\subsection{Standard cases: jump interface and ramp interface}

For reference, the expressions of $R$ in the case $\alpha = 0$ (jump discontinuity) and $\alpha = 1$ (ramp) are known explicitly:
\begin{itemize}
\item When $\alpha = 0$,
\be R=\frac{\sqrt{(\frac{c_-}{c_+})^2-\eta_2^2-\eta_3^2}-\sqrt{1-\eta_2^2-\eta_3^2}}{\sqrt{(\frac{c_-}{c_+})^2-\eta_2^2-\eta_3^2}+\sqrt{1-\eta_2^2-\eta_3^2}}
\label{reflection-sharp}\ee
There is no $\theta$ parameter in the case $\alpha = 0$, hence this expression cannot conform to the setup of our main result when $\alpha \rightarrow 0$.
\item When $\alpha = 1$, introduce $\beta=\theta^{-\frac23}$, and $w_{\pm}(X)= Ai(e^{\pm i\frac{\pi}{3}}X)$. The exact expression of $R$ is known
$$R=\frac{w'_+(\beta)+i\beta^{\frac12}w_+(\beta)}{-w'_+(\beta)+i\beta^{\frac12}w_+(\beta)},$$
and its asymptotic expansion in $\theta$, thanks to classical asymptotics of the Airy function,  is 

$$R=\frac{i}{8}\theta+O(\theta^2)=\frac{1}{(2i)^3}\theta+O(\theta^2),$$
which matches the general expression (See Proposition \ref{plusinfini} and Lemma \ref{estimation-layer}).\end{itemize}

\subsection{Related work}

Our construction similar to the decomposition into incoming and outgoing waves introduced by Jost in \cite{Jost}, and also mentioned in \cite{JostPais}. They are sometimes called {\it irregular or Jost solutions}, and for the radial Schr\"odinger equation $-\psi''+V(r)\psi= k^2\psi$, are $f_{\pm}= e^{\pm i k r}+o(1)$ as $r\rightarrow +\infty$. They serve as a base of decomposition of all solutions of the radial Schr\"odinger equation. Indeed, with the words {\it einlaufende kugelwelle} (entering spherical wave ) $f(k,r) \simeq e^{-ik.r}$ ((4) \cite{Jost}), and the remark that $f(k,.)$ and $f(-k, .)$ form a fundamental system for the ODE. The name ``outgoing spherical wave" appears in the companion paper of Jost and Pais \cite{JostPais}.

We use extensively the notion of limiting absorption principle, or limiting amplitude principle, for the definition of the outgoing and incoming waves. This notion was introduced by A. G. Sveshnikov in \cite{S1950}, who considers a wave number $k_1$ such that $k_1^2=k^2+i\epsilon$.  Sveshnikov himself refers to a paper of Ignatovski \cite{I}; in this paper of 1905, the complex perturbation was not arbitrary but was given by physical absorption. Complexification of the wave number of course also underlies the ``$i \epsilon$" prescription of the propagators for the Helmholtz and Schr\"{o}dinger equations \cite{Folland}.

%
%
%
%

As for discontinuities of the velocity of the wave, the case of the ramp is addressed by L.M. Brekhovskikh and O.A. Godin \cite{Godin} (see Chap. 3, expression (3.5.1) for example). Nevertheless, they treat only acoustic wave equations in the case where one has a representation in terms of special functions, and no nonsmooth profile is otherwise considered.

In \cite{WA1, WA2}, K. Wapenaar quantifies the reflection coefficient for the so-called self-similar profile, when $c(x) = c_1 (-x)^\alpha$ for $x<0$ and $c(x)=c_2 x^\alpha$ for $x > 0$. His analysis relies on a self-similar change of variables -- which is not directly useful in the context of this paper -- and an asymptotic study of the transfer matrix linking the solution and its derivative across the interface.

More recently, K. Wapenaar \cite{WA} derives a decomposition of the system defining waves in which there is a preferred direction of propagation in a set of coupled equations for waves propagation in the opposite directions along this preferred axis and derive the associated Green's functions.

As we put the final touches to this paper, we learned of the work
of O. Gannot and J. Wunsch  \cite{Wu} which establishes a general result, for $\alpha>0$, on the
propagation of singularities theorem for the Schr\"{o}dinger equation with a
fractional jump of the potential on a conormal surface, with regularity
improvements for $\alpha>1$. They deduce an estimate on the reflection
coefficient which is close to ours, in the particular case of the
potential $x_+^{\alpha}-1$ for $x\in (-\infty, x_0)$ in a one-dimensional setup.
Their analysis relies on a microlocal study of the bicharacteristics
near the jump, and does not deal with the behavior at infinity -- which differs in wave vs Schr\"{o}dinger equations and is
one of the key points of our study. Their analysis also does not seem to immediately give a precise estimate of the reflection coefficient in the multi-d set-up.




\section{Definitions and setup}

\subsection{Nondimensionalization}

We assume constant density, acoustic waves propagating in a heterogeneous wave speed in 3 space dimensions (\ref{wave-eq}),
$$\left( \frac{1}{c^2(\mathbf{x})} \pd_{t^2} - \Delta_{\mathbf{x}} \right) u = 0, \qquad \mathbf{x} = (x_1,x_2,x_3) \in \R^3.
$$
The model we consider for the wave speed is the fractional ramp (\ref{eq:c_lambda}).

For completeness, we also consider the more general inhomogeneous form 
\be\label{eq:c_lambdaofx}
c^{-2}(\mathbf{x}) = c_0^{-2} \left[ 1 + \lambda \left( \left( \frac{x_1}{\ell} \right)_{+}^\alpha \right) \right],
\ee
where $\lambda(x) = \int_0^x \chi(y) dy$, for some positive, compactly supported, $\chi \in C^\infty(\R)$ equal to $1$ in a neighborhood of $x = 0$. The proof assumes the case $\lambda(x) = x$, and treats the modification in the appendix.



Assume a harmonic time dependence of the form $e^{i \omega t}$, i.e., take a Fourier transform in $t$. After taking another partial Fourier transform in the transverse coordinates $(x_2,x_3)$, and introducing $\eta= c_0 \omega^{-1}(k_2^2+k_3^2)^{\frac12}$, the wave equation becomes
\[
\left( \pd_{x_1}^2 + \omega^2c_0^{-2} \left[ 1 + \left( \frac{x_1}{\ell} \right)_{+}^\alpha - \eta^2 \right] \right) \hat{u} = 0,
\]
with boundedness conditions at $x_1 = \pm \infty$. The continuity of $\hat{u}$ and $\pd_{x_1} \hat{u}$ at $x_1 = 0$ owes to the boundedness of the wave speed $c(\mathbf{x})$. 

The equation is further non-dimensionalized by letting
\[
x_1 = \eps x, \qquad \eps = \frac{c_0}{\omega} (1 - \eta^2)^{-1/2}.
\]
We can understand $\eps$ as the horizontal wavelength (divided by $2 \pi$), when $x_1$ is in a horizontal direction.

Freeing up $u$ to let $u(x) := \hat{u}(x_1)$, and freeing up $c$ to let 
\be\label{c-fractional}
c^{-2}(x) := 1 + \theta \, x_{+}^\alpha,
\ee
with
\beq\label{eq:theta}
\theta = \left( \frac{c_0}{\ell \omega} \right)^{\alpha}  (1 - \eta^2)^{- \frac{\alpha + 2}{2}},
\eeq
we obtain the model equation
\be\label{eq:model-problem}
\left( \pd^2_x + c^{-2}(x) \right) \, u = 0.
\ee

Note that the cosine of the incidence angle is $\cos \varphi_i = (1 - \eta^2)^{1/2}$. We can therefore rewrite $\theta$ as
\[
\theta = \left( \frac{\eps}{\ell} \right)^{\alpha} \frac{1}{(\cos \varphi_i)^{\alpha + 2}}.
\]
This form of $\theta$ reveals the importance of the non-dimensional ratio $\eps / \ell$ of the horizontal wavelength by the skin depth. We see that regimes of small $\theta$ correspond to either
\begin{itemize}
\item weak strength of the fractional reflector (large $\ell$); and/or
\item high frequencies (small $\eps)$.
\end{itemize}
Furthermore, this regime is only possible for non-grazing waves ($\cos \varphi_i \ne 0$).

The physical dimensions are gathered in the following table, where $m$ is meter, $s$ is second, and $1$ is dimensionless.

\begin{center}
\begin{tabular}{| l | l |}
\hline
$x_j, \eps, \ell, k_j^{-1}$ & $m$  \\ 
\hline
$t, \omega^{-1}$ & $s$ \\
\hline
$c_0$ & $ms^{-1}$ \\ 
\hline
$x, c, \theta, \eta$ & $1$ \\
\hline
\end{tabular}
\end{center}

Our objective is now to define and estimate the reflection coefficient for (\ref{eq:model-problem}), in terms of $\theta$.

\subsection{WKB approximate solutions}

An important analytical tool in this paper is the WKB construction of approximate solutions,
\be\label{model-solution}
v^>(x) = b(x) e^{- i \phi(x)}, \qquad v^<(x) = b(x) e^{i \phi(x)},
\ee

with 
\[
\phi(x) = \int_0^x \frac{1}{c(y)} \, dy, \qquad b(x) = c^{\frac{1}{2}}(x).
\]
This choice coincides with the first term of the usual WKB expansion; writing the next terms would require higher degrees of differentiability of $c(x)$. In the sequel, we consider $\alpha>0$ in the fractional ramp problem using $c \in C^0(\R) \cap  C^{2}(]0, +\infty[)$, on $[x_0, +\infty[$, $x_0>0$ to be chosen later.

An important property of the phase $\phi$ is that $\lim_{x \to \pm \infty} \phi(x) = \pm \infty$ with $\phi' > 0$. The identity $\phi'' + 2 (c^{1/2})' \phi' = 0$ is also useful. The two functions $v^>$ and $v^<$ are not solutions of (\ref{eq:model-problem}), but they solve the modified equation
\be\label{eq:WKB}
v'' + c^{-2} v = c^{- \frac{1}{2}} \left( c^{\frac{1}{2}} \right)''  v.
\ee


The functions $v^>$ and $v^<$ have the interpretation of right-going and left-going waves, respectively. This property is apparent when restoring the time dependence $e^{i \omega t}$. This interpretation can also be formalized by considering the extended equation 
\be\label{eq:WKB-ext}
\left( \pd_x^2 + (1 + i \sigma)^2 c^{-2} \right) \, v(x ; \sigma) = c^{- \frac{1}{2}} \left( c^{\frac{1}{2}} \right)''  v(x;\sigma),
\ee
with a small parameter $\sigma \in \R$. The solutions are now 
\be\label{model-solution-sigma}
v^>(x;\sigma) = b(x) e^{- i (1+ i \sigma) \phi(x)}, \qquad v^<(x;\sigma) = b(x) e^{i (1+ i \sigma)\phi(x)}.
\ee
The left-going or right-going character, or ``polarization", is now tied to the direction in which exponential decay occurs as a function of the sign of $\sigma$. In case $\pm \sigma > 0$, we have
\[
\lim_{x \to \mp \infty} v^>(x;\sigma) = 0, \qquad \lim_{x \to \pm \infty} v^<(x;\sigma) = 0.
\]
These expressions are used in the next section, as the basis for the definition of polarized waves for the model equation (\ref{eq:model-problem}).

For future convenience we introduce the notation
\be\label{def:M}
M(x) = c^{\frac{1}{2}} \left( c^{\frac{1}{2}} \right)'' = \left( (v^{>})''+ c^{-2} v^{>} \right) \, v^{<},
\ee
so (\ref{eq:WKB}) can alternatively be written as
\[
v'' + c^{-2} v = M c^{-1}  v,
\]
and similarly for (\ref{eq:WKB-ext}),
\be\label{eq:WKB-ext2}
\left( \partial^2_{x} + (1+ i \sigma)^2 c^{-2}(x) \right) v(x;\sigma) - M c^{-1} v(x;\sigma) = 0.
\ee
Many properties of the function $M(x)$ valid for all $\alpha > 0$ are listed and proven in Appendix \ref{sec:appendixB}, including the fact that it is $O(\theta)$, and that it is locally integrable at $x = + \infty$.

Also note in passing that $c^{-1} v^{>} v^{<} = 1$.

Finally, we also introduce the zero-th order WKB approximations
\begin{equation}\label{eq:w}
w^{>}(x) = e^{-i \phi(x)}, \qquad w^{<}(x) = e^{i \phi(x)}.
\end{equation}

\subsection{Incoming and outgoing waves}

The notion of reflection coefficient involves a comparison of waves polarized as incoming and outgoing. As in the previous section, we introduce a small parameter $\sigma \in \R$ and consider the extended problem
\be\label{eq:model-sigma}
\left( \partial^2_{x} + (1+ i \sigma)^2 c^{-2}(x) \right) u(x;\sigma)=0.
\ee

In contrast to the WKB functions $v^{>}$ and $v^{<}$, we now have four polarized solutions. Each is defined up to a multiplicative constant.

\bigskip
\begin{definition}\label{def:polarized-waves}
A nonzero solution $u$ of (\ref{eq:model-problem}) is said to be
\begin{itemize}
\item outgoing to $+\infty$, if there exists a sequence of solutions $u(x;\sigma)$ of (\ref{eq:model-sigma}) with $\sigma < 0$ such that\footnote{The $\sigma \to 0$ limits are all understood to converge uniformly over compact sets of $x \in \R$.}
\[
\lim_{\sigma \to 0^{-}} u(x;\sigma) = u(x), \qquad \lim_{x \to +\infty} u(x; \sigma) = 0.
\]
The space of such solutions is denoted by ${\mathcal U}^{>}_{+\infty}$.
\item incoming from $+\infty$, if there exists a sequence of solutions $u(x;\sigma)$ of (\ref{eq:model-sigma}) with $\sigma > 0$ such that
\[
\lim_{\sigma \to 0^{+}} u(x;\sigma) = u(x), \qquad \lim_{x \to +\infty} u(x; \sigma) = 0.
\]
The space of such solutions is denoted by ${\mathcal U}^{<}_{+\infty}$.
\item outgoing to $-\infty$, if there exists a sequence of solutions $u(x;\sigma)$ of (\ref{eq:model-sigma}) with $\sigma < 0$ such that
\[
\lim_{\sigma \to 0^{+}} u(x;\sigma) = u(x), \qquad \lim_{x \to -\infty} u(x; \sigma) = 0.
\]
The space of such solutions is denoted by ${\mathcal U}^{<}_{-\infty}$.
\item incoming from $-\infty$, if there exists a sequence of solutions $u(x;\sigma)$ of (\ref{eq:model-sigma}) with $\sigma > 0$ such that
\[
\lim_{\sigma \to 0^{-}} u(x;\sigma) = u(x), \qquad \lim_{x \to -\infty} u(x; \sigma) = 0.
\]
The space of such solutions is denoted by ${\mathcal U}^{>}_{-\infty}$.
\end{itemize}

\end{definition}

We interpret any element $u \in {\mathcal U}^{>}_{+\infty}$ as a transmitted wave, $u\in {\mathcal U}^{<}_{-\infty}$ as a reflected wave, and $u\in {\mathcal U}^{>}_{-\infty}$ as an incident wave (from the left, as is the case throughout the paper).  Using a classical result on asymptotic behavior of systems of ODEs, we have the
\begin{proposition}\label{teo:coddington}

\begin{itemize}
\item[(i)] Each of the four subspaces ${\mathcal U}^{>}_{\pm \infty}$, ${\mathcal U}^{<}_{\pm \infty}$ has dimension 1.  
\item[(ii)] Any pair $(u_1, u_2)$ of nonzero solutions with $u_1 \in {\mathcal U}^{>}_{- \infty}$ and $u_2 \in {\mathcal U}^{<}_{- \infty}$; or $u_1 \in {\mathcal U}^{>}_{+ \infty}$ and $u_2 \in {\mathcal U}^{<}_{+ \infty}$, forms a fundamental system for (\ref{eq:model-problem}). 
\item[(iii)] Any nonzero element $u \in {\mathcal U}^{>}_{\pm \infty}$ admits a nonzero, finite limit $\lim_{x \to \pm \infty} \frac{u(x)}{v^{>}(x)}$. Correspondingly, any nonzero element $u \in {\mathcal U}^{<}_{\pm \infty}$ admits a nonzero, finite limit $\lim_{x \to \pm \infty} \frac{u(x)}{v^{<}(x)}$.
\item[(iv)] $\lim_{x \to \infty} u(x;\sigma) = 0$ if and only if $\| u(\cdot; \sigma) \|_\infty < \infty$.
\end{itemize}
\end{proposition}

We leave the proof of this result in Appendix \ref{sec:appendixA}.

This proposition provides a natural way of normalizing the polarized solutions: 
\begin{itemize}
\item we let $u^{>}_{\pm \infty}(x)$ for the element $u$ of ${\mathcal U}^{>}_{\pm \infty}$ such that $\lim_{x \to \pm \infty} \frac{u(x)}{v^{>}(x)} = 1$; and
\item we let $u^{<}_{\pm \infty}(x)$ for the element $u$ of ${\mathcal U}^{<}_{\pm \infty}$ such that $\lim_{x \to \pm \infty} \frac{u(x)}{v^{<}(x)} = 1$.
\end{itemize}


In the familiar case of a plane wave, when $\theta = 0$, we have $u^{>}_{-\infty}(x) = u^{>}_{+\infty}(x) = e^{-ix}$ and $u^{<}_{-\infty}(x) = u^{<}_{+\infty}(x) = e^{ix}$, but in a heterogeneous medium the four polarized waves are in general distinct. In the fractional ramp example, where $c^{-2}(x) = 1 + \theta x_{+}^\alpha$, or in any case where the medium is homogeneous uniform in $x<0$, we still have\begin{equation}\label{eq:expi}
u^{>}_{- \infty}(x) = v^{>}(x) = e^{- ix}, \qquad u^{<}_{- \infty}(x) = v^{<}(x) = e^{ ix}, \qquad x < 0.
\end{equation}
For illustration, in the case $\alpha = 1$, it is shown in Section \ref{sec-ramp} that 
\[
u^{>}_{+\infty}(x) = \overline{u^{<}_{+\infty}(x)} = 2 \pi^{1/2} e^{i \frac{\pi}{12}} \theta^{-\frac{1}{6}} Ai(e^{i\frac{\pi}{3}}\theta^{\frac13}(x+\theta^{-1})), \qquad x > 0.
\]

The transmitted wave is then expanded as
\begin{equation}
\label{coeffAB}
u^{>}_{+ \infty} = A \, u^{<}_{-\infty} + B \, u^{>}_{-\infty},
\end{equation}
and the reflection coefficient defined as $R = A/B$. Notice that the value of $R$ does not depend on the choice of normalization of $u^{>}_{+ \infty}$. The transmission coefficient is $1/B$, and does depend on this choice.

It is easy to see from (\ref{eq:expi}), and from continuity of $u$ and $u'$ near $x=0$, that for $u \in {\mathcal U}^{>}_{+ \infty}$, the reflection coefficient can be determined from
\begin{equation}\label{eq:qR}
\frac{u'(0)}{u(0)} = i\frac{R-1}{R+1}.
\end{equation}
This relation is key to computing $R$ in the sequel. The case $R = -1$ corresponds to $u(0) = 0$, but we will see in the sequel that an assumption of small $\theta$ prevents this scenario. The quantity $\frac{R-1}{R+1}$ can be seen as a nondimensionalized impedance.

We can now remind the reader of our main result (Theorem \ref{teo:main} in the introduction).

\begin{theorem}
For all $\alpha > 0$, as $\theta \to 0$,
\[
R = \frac{\Gamma(\alpha+1)}{(2i)^{\alpha+2}} \, \theta  + O \left( \theta^{1 + \min \left( \frac{1}{\alpha}, 1 \right) } \right).
\]
\end{theorem}

\subsection{First steps and architecture of the proof}

Section \ref{sec:Volterra} contains preparatory material. The core of the proof is in Section \ref{sec-limit-R}. It establishes that the remainder is a $o(\theta)$. Section \ref{sec-limit-theta} establishes the more precise form quoted above for this remainder. Although the proof covers all cases $\alpha \geq 0$, Section \ref{sec-ramp} gives an explicit (exact, non-asymptotic) formula for $R$ that matches Theorem \ref{teo:main} in the special case $\alpha = 1$ (available from special function analysis).

The main idea of the proof is to express any nonzero solution that is outgoing to $+\infty$, denoted $u$ for brevity, and defined up to an unimportant nonzero multiplicative scalar, in terms of the WKB functions $v^{>}$ and $v^{<}$. This requires writing a Volterra equation, and solving it iteratively. This is performed on $[x_0, +\infty)$, for an arbitrary $x_0$, such that the Volterra equation has a contracting kernel when $\theta < \theta_0(x_0)$. As a consequence of the contractibility of this kernel, we obtain $u(x_0) \ne 0$. This argument provides a value of $q =\frac{u'(x_0)}{u(x_0)}$, which is then linked to $\frac{u'(0)}{u(0)}$ via a Cauchy problem on $(u(x), u'(x))$, for $x \in [0,x_0]$, with data at $x_0$ equal to $(1,q)$. The reflection coefficient is then deduced from (\ref{eq:qR}).

For $x \in [0,x_0]$, it is clear that this Cauchy problem is well-posed, but it is advantageous to express its solution in terms of $w^{>}(x):=e^{-i\phi(x)}$ and $w^{<}(x):=e^{i\phi(x)}$, via
\[
\begin{pmatrix} u(x) \\ u'(x) \end{pmatrix} = a^{>}(x) \begin{pmatrix} w^{>}(x) \\ (w^{>})'(x) \end{pmatrix} + a^{<}(x) \begin{pmatrix} w^{<}(x) \\ (w^{<})'(x) \end{pmatrix}
\]
(which makes use of $a^{>}(x)(w^{>})'((x)+a^{<}(x)w^{<})'(x)=0$). One then deduces $k$ such that $a^{<}(x_0)=k a^{>}(x_0)$ through the equality $q (1+ke^{2i\phi(x_0)})=i\phi'(x_0)(-1+ke^{2i\phi(x_0)})$.

For $x \in [x_0, \infty)$, it is preferable to express $u$ in terms of $v^{<}$ and $v^{>}$. Equation (\ref{eq:WKB-ext2}) is a homogeneous linear differential equation for which $v^{>}(\cdot; \sigma)$ and $v^{<}(\cdot; \sigma)$ is a generating pair of solutions. Since equation (\ref{eq:model-sigma}) can be seen as a inhomogeneous version of (\ref{eq:WKB-ext2}), where the right-hand side is $- M c^{-1} v$, any solution of  (\ref{eq:model-sigma}) can be sought as a combination of
$v^{>}(.;\sigma)$ and $v^{<}(.;\sigma)$ via Duhamel's principle (the method of
variation of parameters) as
\be\label{eq:Duhamel}
u(x;\sigma) = A_{\sigma}(x) \, v^{>}(x;\sigma) + B_{\sigma}(x) \, v^{<}(x;\sigma).
\ee

In a standard fashion, we additionally impose 
\be\label{eq:condDuhamel} 
(A_{\sigma}(x))' \, v^{>}(x;\sigma) + (B_{\sigma}(x))' \, v^{<}(x;\sigma)=0,\ee
so that
\be\label{eq:Duhamelprime}
u'(x;\sigma) = A_{\sigma}(x) \, (v^{>})'(x;\sigma) + B_{\sigma}(x) \, (v^{<})'(x;\sigma).
\ee
With the Wronskian equal to $2ib^2(x)c^{-1}(x)=2i$, equations (\ref{eq:Duhamel}) and (\ref{eq:Duhamelprime}) determine the coefficients $A_{\sigma}(x)$ and $B_{\sigma}(x)$ uniquely as
\be\label{coefficients}
\left\{
\begin{array}{ll}A_{\sigma}(x) =& \frac{1}{2i(1+i\sigma)}[u(x;\sigma)(v^{<})'(x;\sigma) -u'(x;\sigma)v^{<}(x;\sigma)]\cr
B_{\sigma}(x) =& \frac{1}{2i(1+i\sigma)}[-u(x;\sigma)(v^{>})'(x;\sigma) +u'(x;\sigma)v^{>}(x;\sigma)]
\end{array}
\right.
\ee

A closed system for $A_{\sigma}$ and $B_{\sigma}$ is obtained from (\ref{eq:condDuhamel}), and by substituting (\ref{eq:Duhamel}) and (\ref{eq:Duhamelprime}) into (\ref{eq:model-sigma}):
$$\bs\ba{l}(A_{\sigma})'(x)v^{>}(x;\sigma)+ (B_{\sigma})'(x)v^{<}(x;\sigma)=0\cr
(A_{\sigma})'(x)(v^{>})'(x;\sigma)+ (B_{\sigma})'x)(v^{<})'(x;\sigma)=-M(x)c^{-1}(x)(A_{\sigma}(x)v^{>}(x;\sigma)+ B_{\sigma}(x)v^{<}(x;\sigma))\ea\es$$
Using the fact that the Wronskian of the two basis functions is $2i(1+i\sigma)$, one obtains
$$\bs\ba{l}(A_{\sigma})'(x)=\frac{1}{2i(1+i \sigma)}M(x)c^{-1}(x)[A_{\sigma}(x)v^{>}(x;\sigma)+ B_{\sigma}(x)v^{<}(x;\sigma)]v^{<}(x;\sigma)\cr
(B_{\sigma})'(x)=-\frac{1}{2i(1+i \sigma)}M(x)c^{-1}(x)[A_{\sigma}(x)v^{>}(x;\sigma)+ B_{\sigma}(x)v^{<}(x;\sigma)]v^{>}(x;\sigma)\ea\es$$
which can be rewritten, using $c^{-1}(x)v^{>}(x;\sigma)v^{<}(x;\sigma)=1$,
$$\bs\ba{l}(A_{\sigma})'(x)=\frac{1}{2i (1+i \sigma)}M(x)[A_{\sigma}(x)+ B_{\sigma}(x)e^{2i(1+i \sigma)\phi(x)}]\cr
(B_{\sigma})'(x)=-\frac{1}{2i(1+i\sigma)}M(x)[A_{\sigma}(x)e^{-2i(1+i\sigma)\phi(x)}+ B_{\sigma}(x)]\ea\es$$

\bigskip

It is now natural to define an auxiliary function
\begin{equation}\label{eq:defS}
S^{>}_{\sigma}(x) = A_{\sigma}(x) e^{- 2i(1+i\sigma) \phi(x)} + B_{\sigma}(x),
\end{equation}
in order to link $u$ to $v^{<}$ via
\be\label{eq:Duhamel2}
u(x;\sigma) = S^{>}_{\sigma}(x) \, v^{<}(x;\sigma).
\ee
The intention is to use this relation when $u$ is outgoing to $+ \infty$ (when $\sigma = 0)$, or when $\lim_{x \to \infty} u(x;\sigma) = 0$ (when $\sigma < 0$, prior to taking the limit $\sigma \to 0^{-}$). There is no typo: {\em we express the wave outgoing to $+ \infty$ in terms of the incoming WKB wave.} The proof will make it clear why this choice of $S$ is necessary, rather than $A_{\sigma}(x)+ B_{\sigma}(x)e^{2i(1+i \sigma)\phi(x)}$.

After integrating the equations for $(A_{\sigma}(x))'$ and $(B_{\sigma}(x))'$, we obtain the functional equation (which depends on $x_0$)
\begin{align}\label{eq:eqS}
S^{>}_{\sigma}(x) = \; &A_{\sigma}(x_0)e^{-2i(1+i\sigma)\phi(x)}+B_{\sigma}(x_0) \\
& + \frac{1}{2i(1+i\sigma)}\int_{x_0}^{x}M(y)S^{>}_{\sigma}(y)(e^{-2i(1+i\sigma)(\phi(x)-\phi(y))}-1)dy, \notag
\end{align}
with two integration constants $A_{\sigma}(x_0)$ and $B_{\sigma}(x_0)$. For $\sigma < 0$, it will be shown in Proposition \ref{teo:Ssigma} that the condition $\lim_{x \to \infty} u(x;\sigma) = 0$ determines $B_\sigma(x_0)$ uniquely as
$$B_{\sigma}(x_0)=\frac{1}{2i(1+i\sigma)}\int_{x_0}^{+\infty}M(y)S^{>}_{\sigma}(y)dy,$$
hence that equation (\ref{eq:eqS}) becomes the Volterra equation for $S^{>}_{\sigma}(x)$:
\begin{equation}
\label{Volterrasigma}S^{>}_{\sigma}(x)= A_{\sigma}(x_0)e^{-2i(1+i\sigma)\phi(x)} + K_{\sigma}(S^{>}_{\sigma})(x),\end{equation}
where the operator $K_\sigma$ is given by 

\begin{equation}\label{def:VolterraK}K_{\sigma}(f)(x)=\frac{1}{2i(1+i\sigma)}[\int_{x_0}^{x}M(y)f(y)e^{-2i(1+i\sigma)(\phi(x)-\phi(y))}dy
+ \int_{x}^{+\infty}M(y)f(y)dy], f\in  C^0_{b}([x_0, \infty)).\end{equation}
A sufficient condition for convergence of the fixed-point iteration for this Volterra equation is the condition
\begin{equation}\label{eq:Mx0}
M_{x_0} := \int_{x_0}^\infty |M(x)| \, dx < 2.
\end{equation}
Note that $M(x)$ also depends on $\theta$;  Lemma \ref{ineq:Mx0} shows that (\ref{eq:Mx0}) is satisfied for all $0 < \theta < \theta_0(x_0)$. In that case, the Volterra equation can be solved as a convergent series
\[
S^{>}_{\sigma}(x) = A_{\sigma}(x_0)\left( \sum_{j \geq 0} s^{>}_{\sigma,j}(x) \right) \, e^{-2 i (1 + i \sigma) \phi(x)},
\]
which defines the functions $s^{>}_{\sigma,j}(x)$.

Denote by $S^{>}$ the unique solution of (\ref{Volterrasigma}) when $\sigma = 0$, and normalized via $A_{\sigma}(x_0) = 1$:
\begin{equation}
\label{eq:S>}
S^{>}(x)=e^{-2i\phi(x)}+K^0(S^{>})(x).
\end{equation} 
It is then shown in Proposition \ref{teo:Vseries} that the Volterra equation (\ref{Volterrasigma}) remains valid for $\sigma = 0$ when the limit $\sigma \to 0^{-}$ is taken, and indeed expresses the desired outgoing solution $u(x)$ as
\[
u(x) = S^{>}(x) \, v^{<}(x),
\]
which is the $\sigma \to 0$ limit of (\ref{eq:Duhamel2}) (again, when $A_{\sigma}(x_0) = 1$).

We are now in a position to present an equivalent reformulation of $u(x)$ in terms of a radiation condition at $x_0 > 0$.

\begin{theorem}
\label{teo:behavior-at-infinity} For all $\alpha>0$ and $x_0\geq 0$, there exists $\theta_0(x_0)$ such that, for all $\theta<\theta_0(x_0)$, the family of solutions of (\ref{eq:model-problem}) outgoing at $+\infty$ is characterized by
$$\frac{u'(x_0)}{u(x_0)}=
\frac{(v^>)'(x_0)+R_{x_0}(v^<)'(x_0)}{v^>(x_0)+R_{x_0}v^<(x_0)}
,$$
where
\begin{equation}\label{Rx0}
R_{x_0}=\frac{1}{2i}\int_{x_0}^{+\infty}M(x')S^{>}(x')dx'.
\end{equation}
\end{theorem}

This radiation condition at $x_0 > 0$ must now be compared with the reference formula for the reflection coefficient $R$, namely 
$$\frac{u'(0)}{u(0)}=i\frac{R-1}{R+1}.$$
As it is written, it may seem not straightforward to see the similarity. Nevertheless, one can rewrite the identity of Theorem \ref{teo:behavior-at-infinity} as
$$\frac{u'(x_0)}{u(x_0)}=\frac{(v^>)'(x_0)}{v^>(x_0)}\frac{1+R_{x_0}\frac{(v^<)'(x_0)}{(v^>)'(x_0)}}{1+R_{x_0}\frac{v^<(x_0)}{v^>(x_0)}}$$
where it can be noticed that, in the constant coefficient case, $\frac{(v^<)'(x_0)}{(v^>)'(x_0)}=-\frac{v^<(x_0)}{v^>(x_0)}=-1$ and $\frac{(v^>)'(x_0)}{v^>(x_0)}=-i$.

The interpretation of the role of $x_0$ is also enlightened by the following remarks:
\begin{itemize}
\item When $\alpha > 1$, the limit as $x_0 \to 0$ exists, is finite, and gives
\[
R=\frac{1}{2i}\int_{0}^{+\infty}M(x')S^{>}(x')dx',
\] 
where $S^{>}(x)$ is now understood to correspond to $x_0 = 0$ in (\ref{eq:VolterraK}) and (\ref{eq:S>}). In fact, the whole argument could have used $x_0 = 0$ in that case.

\item When $\alpha < 1$, the limit cannot be taken in the same fashion. The operator $K_{\sigma}$ can no longer be defined by (\ref{def:VolterraK}) in the limit $x_0 \to 0$, due to the lack of integrability of $M(x)$ at $x = 0$. The quantity $R_{x_0}$ also diverges as $x_0 \to 0$. Instead, we consider the Cauchy problem on $[0,x_0]$, and find an expression to relate $R$ to $R_{x_0}$. Since this more general argument also applies to the case $\alpha > 1$, we do not differentiate the two cases in the proof in Section \ref{sec-limit-R}.

\end{itemize}

The limit of $\frac{R}{\theta}$ is finally considered when $\theta\rightarrow 0$, which establishes Theorem \ref{teo:main}. This is the aim of Section \ref{sec-limit-R}. Precise estimates on $\frac{R}{\theta}-\frac{\Gamma(\alpha+1)}{(2i)^{\alpha+2}}$ and on the function $\Gamma$ are given in Section \ref{sec-limit-theta}.

\section{Volterra series and radiation condition}\label{sec:Volterra}
In this section, we establish the Volterra equation for $S^{>}_{\sigma}(x)$, the convergence properties of the corresponding Volterra series. We also prove Theorem \ref{teo:behavior-at-infinity}.
 
Recall that $S^{>}_{\sigma}(x)$ is defined by (\ref{eq:defS}) or (\ref{eq:Duhamel2}), without restriction on $\sigma \in \R$, in relation to any solution $u(x;\sigma)$ of (\ref{eq:model-sigma}), and that it obeys equation (\ref{eq:eqS}) with two arbitrary constants $A_{\sigma}(x_0)$ and $B_{\sigma}(x_0)$. In the case $\sigma < 0$, we start by fixing $B_{\sigma}(x_0)$ in equation (\ref{eq:eqS}), from imposing the condition that $u(x;\sigma) \to 0$ as $x \to \infty$ (which is equivalent to $u(x;\sigma)$ uniformly bounded on $[x_0, \infty)$ by (iv) of Proposition \ref{coddington}.) The result is a Volterra equation on the half-line $[x_0, \infty)$, which involves the operator $K_{\sigma}$ (Definition \ref{def:VolterraK}):

\begin{proposition}
\label{teo:Ssigma}
Assume $x_0>0$ and $\sigma < 0$. When $u(x;\sigma)$ is uniformly bounded for $x \in [x_0, \infty)$, then the function $S^{>}_{\sigma}$ given by (\ref{eq:Duhamel2}) obeys the Volterra equation
\begin{equation}\label{eq:VolterraK}
S^{>}_{\sigma}(x)= A_{\sigma}(x_0)e^{-2i(1+i\sigma)\phi(x)}+K_{\sigma}(S^{>}_{\sigma})(x).
\end{equation}
\end{proposition}

\begin{proof}[Proof of proposition \ref{teo:Ssigma}]

Fix $\sigma<0$. If $u(x;\sigma)$ is bounded on $[x_0, \infty)$, then $S^{>}_{\sigma}(x)=\frac{u(x;\sigma)}{v^{<}_{\sigma}(x)}$ tends to 0 when $x \to +\infty$, thanks to the inequality 
$$\vert \frac{u(x;\sigma)}{v^{<}_{\sigma}(x)}\vert\leq b^{-1}(x) e^{\sigma \phi(x)} \max_{x_0 \leq x <\infty} |u(x;\sigma)|,$$
and the fact that $b^{-1}(x)e^{\sigma \phi(x)}=(1+\theta x^\alpha)^{\frac14}e^{\sigma\int_0^x(1+\theta y^{\alpha})^{\frac12}dy}$ tends to zero\footnote{This is also the case for the more general form $\lambda(x) \ne x$.} when $x\rightarrow +\infty$ for $\sigma<0$. It is thus bounded on $[x_0, +\infty)$. Since (i) $S^{>}_{\sigma}$ is bounded, (ii) $\int_{x_0}^{+\infty}\vert M(y)\vert dy<+\infty$, (iii) $\vert e^{-2i(1+i\sigma)(\phi(x)-\phi(y))}-1\vert\leq 2$ for $y\leq x$, and (iv) $e^{2 \sigma \phi(x)} \to 0$ for $x \to \infty$, the conditions of Lebesgue's dominated convergence theorem are met and we obtain
 $$\lim_{x \to + \infty} \int_{x_0}^{x}M(y)S^{>}_{\sigma}(y) \left( e^{-2i(1+i\sigma)(\phi(x)-\phi(y))}-1 \right) dy =  -\int_{x_0}^{+\infty}M(y)S^{>}_{\sigma}(y)dy.$$
Hence a necessary condition for $\lim_{x \to + \infty} S^{>}_{\sigma}(x) = 0$, in view of equation (\ref{eq:eqS}), is that 
$$B_{\sigma}(x_0)=\frac{1}{2i(1+i\sigma)}\int_{x_0}^{+\infty}M(y)S^{>}_{\sigma}(y)dy.$$
Equation (\ref{eq:eqS}) thus becomes
\begin{align}\label{eq-fonct}
S^{>}_{\sigma}(x)&= A_{\sigma}(x_0)e^{-2i(1+i\sigma)\phi(x)}+\frac{1}{2i(1+i\sigma)}[\int_{x_0}^{x}M(y)S^{>}_{\sigma}(y)e^{-2i(1+i\sigma)(\phi(x)-\phi(y))}dy \\
&+ \int_{x}^{+\infty}M(y)S^{>}_{\sigma}(y)dy], \notag
\end{align}
or simply $S^{>}_{\sigma}(x)= A_{\sigma}(x_0)e^{-2i(1+i\sigma)\phi(x)}+K_{\sigma}(S^{>}_{\sigma})(x).$
\end{proof}

Note in passing that the proof argument with Lebesgue's theorem would not have been possible if $S^{>}_{\sigma}$ had been defined from $v^{>}$ instead of $v^{<}$ in equation (\ref{eq:Duhamel2}).

The Volterra equation (\ref{eq:VolterraK}) can now be considered in the case $\sigma = 0$ as well. The following result shows that $M_{x_0} < 2$, as in equation (\ref{eq:Mx0}), is a sufficient condition under which $K_\sigma$ is a contraction, hence for which the equation can be solved by iteration.

\begin{proposition}\label{teo:Vseries}
Let $x_0 > 0$ and $\theta > 0$, such that $M_{x_0} < 2$, and let $\sigma \leq 0$.
\begin{itemize}
\item[(i)] The operator $K_{\sigma}$ is a contraction, and satisfies the functional inequality
\[
\vert\vert K_{\sigma}(f)(x)\vert\vert_{L^\infty(x_0,\infty)} \leq \frac{M_{x_0}}{2} \, \vert\vert f\vert\vert_{L^\infty(x_0,\infty)}.
\]
\item[(ii)] The unique solution to equation (\ref{eq:VolterraK}) is $A_{\sigma}(x_0)S^{>}_{\sigma}(x)$, where 
\[
S^{>}_\sigma(x) =  \sum_{n=0}^{+\infty}S^{>}_{\sigma,n}(x),
\]
where $S^{>}_{\sigma,n}(x)$ are defined in sequence as
\be\label{sequences}\ba{l}S^{>}_{\sigma,0}(x)= e^{-2i(1+i\sigma)\phi(x)},\cr
S^{>}_{\sigma,n+1}(x)=K_{\sigma}(S^{>}_{\sigma,n})(x), \qquad n\geq 0\ea\ee
and satisfy
$$\vert\vert S^{>}_{\sigma,n}(x)\vert\vert_{L^{\infty}(x_0, +\infty)}\leq \left( \frac{M_{x_0}}{2} \right)^n.$$

Furthermore, introduce $s^{>}_{\sigma,n}=S^{>}_{\sigma,n}(x)e^{2i(1+i\sigma)\phi(x)}$. One has
\[
S^{>}_{\sigma}(x) =  \left( \sum_{n=0}^{+\infty}s^{>}_{\sigma,n}(x) \right) \, e^{-2i(1+i\sigma)\phi(x)},
\]
with
\[
\vert\vert s^{>}_{\sigma,n}(x)\vert\vert_{L^{\infty}(x_0,+\infty)}\leq \left( \frac{M_{x_0}}{2} \right)^n
\]
and
\[
\vert\vert \sum_{n=N+1}^{+\infty}s^{>}_{\sigma,n}(x)\vert\vert_{L^{\infty}(x_0,+\infty)}\leq \frac{1}{1-\frac{M_{x_0}}{2}}\left( \frac{M_{x_0}}{2} \right)^{N+1}
\]
\item[(iii)] For each $\sigma<0$, the family of solutions of (\ref{eq:model-sigma}), bounded for each $\sigma \leq 0$ when $x\rightarrow +\infty$, is given by
\begin{equation}
\label{solutionbornee}
u(x;\sigma)= A_{\sigma}(x_0) \, \left( \sum_{n=0}^{+\infty}s^{>}_{\sigma,n}(x) \right) \; b(x) e^{-i(1+i\sigma)\phi(x)}.
\end{equation}
\end{itemize}
\end{proposition}

\begin{proof}[Proof of Proposition \ref{teo:Vseries}]
The proof of (i) is a simple calculation. For (ii), consider the sequence $S^{>}_{\sigma,n}$ as defined in the proposition. As $K_{\sigma}$ is a contraction, the series $\sum_{n\geq 0}S^{>}_{\sigma,n}(x)$ is uniformly convergent on $[x_0, +\infty)$. This implies that the partial sum $S^{>}_{\sigma,N}= \sum_{k=0}^N S^{>}_{\sigma,k}$ converges uniformly towards a continuous function as $N\rightarrow +\infty$.

Note that the inequality on $S^{>}_{\sigma,n}$  does not imply the inequality on $s^{>}_{\sigma,n}$ ($s^{>}_{\sigma,n}= e^{2i(1+i\sigma)\phi(x)}S^{>}_{\sigma,n}$ implies $\vert s^{>}_{\sigma,n}(x)\vert \leq \left(\frac{M_{x_0}}{2}\right)^ne^{-2\sigma\phi(x)}$, which is not bounded for $\sigma<0$). The equality which leads to the estimate on $s^{>}_{\sigma,n}$ is the sequence
$$s^{>}_{\sigma,n+1}(x)=\frac{1}{2i}[e^{-2i(1+i\sigma)\phi(x)}\int_{x_0}^xM(y)s^{>}_{\sigma,n}(y)dy+\int_{x}^{+\infty}M(y)s^{>}_{\sigma,n} dy]$$
which imply
$$\vert s^{>}_{\sigma,n+1}(x)\vert \leq\frac{1}{2}[e^{2\sigma\phi(x)}\int_{x_0}^x\vert M(y)\vert \vert s^{>}_{\sigma,n}(y)\vert dy+\int_{x}^{+\infty}\vert M(y)\vert \vert s^{>}_{\sigma,n}(y)\vert dy]\leq \frac{M_{x_0}}{2}\mbox{max}_{[x_0, +\infty)}\vert s^{>}_{\sigma,n}(y)\vert.$$
It is easy to deduce, by iteration, the estimate on $s^{>}_{\sigma,n}$ and to deduce from the estimate on $s^{>}_{\sigma,n}$ the estimate on $\sum_{n\geq N+1} s^{>}_{\sigma,n}$.

Item (iii) is a consequence of $v^{<}_{\sigma}(x)=b(x)e^{i(1+i\sigma)\phi(x)}$ and the fact that if $u(x,\sigma)$ is a bounded solution of (\ref{eq:model-sigma}), then $u(x,\sigma)=S^{>}_{\sigma}(x)v^{<}_{\sigma}(x)$. As $A_{\sigma}(x_0) \, \left( \sum_{n=0}^{+\infty}s^{>}_{\sigma,n}(x) \right) \; b(x) e^{-i(1+i\sigma)\phi(x)}$ is bounded on $[x_0, +\infty)$, we proved that $u(x;\sigma)$ obtained by this procedure is bounded.

The proof of (iii) is a consequence of the uniform bound $\vert \sum_{n=0}^{+\infty}s^{>}_{\sigma,n}(x)\vert\leq \frac{1}{1-\frac{M_{x_0}}{2}}$ and $b\leq 1, \phi \geq 0, \sigma <0$. Using the previous estimate, one obtains
$$\vert u(x;\sigma)\vert\leq \vert A_{\sigma}(x_0) \vert\frac{1}{1-\frac{M_{x_0}}{2}}e^{\sigma \phi(x)}b(x)\leq \vert A_{\sigma}(x_0) \vert\frac{1}{1-\frac{M_{x_0}}{2}}$$
thanks to $b(x)\leq 1$, $\phi(x)\geq 0$ and $\sigma<0$. Hence $u$ given by (\ref{solutionbornee}) is bounded. The proposition is proven.

Note that in the third item one must consider $\sigma<0$ because, for $\sigma=0$, all solutions of (\ref{eq:WKB-ext2}) are bounded.

\end{proof}

For illustration, the first two terms are
\begin{align*}
s^{>}_{\sigma,0}(x)&=1, \\
s^{>}_{\sigma,1}(x)&=\frac{1}{2i(1+i\sigma)}[\int_{x_0}^xM(y)dy+\int_x^{+\infty}M(y)e^{2i(1+i\sigma)(\phi(x)-\phi(y))}dy], \\
\ldots
\end{align*}

\begin{proposition}
\label{teo:S0}
Let $x_0 > 0$ and $\theta > 0$, such that $M_{x_0} < 2$.	
\begin{enumerate}
\item
The family of solutions of (\ref{eq:model-problem}), outgoing at $+\infty$, is given by a constant times the nonzero function
$$S^{>}(x)v^{<}(x),$$
where $S^{>}$ solves (\ref{eq:S>}).
\item The function $S^{>}(x)e^{2i\phi(x)}$ has a uniformly convergent expansion
\[
S^{>}(x)e^{2i\phi(x)}= 1+\sum_{j\geq 1}s^{>}_{j}(x)
\]
\item The inequalities of Proposition \ref{teo:Vseries} extend to the case $\sigma=0$.
\end{enumerate}
\end{proposition}

\begin{proof}[Proof of Proposition \ref{teo:S0}]

Prove first that $S^{>}v^{<}$ is nonzero. One checks the two relations
$$(S^{>}v^{<})(x_0)=b(x_0)(e^{-i\phi(x_0)}+R_{x_0}e^{i\phi(x_0)}), (S^{>}v^{<})'(x_0)=b'(x_0)(S^{>}v^{<})(x_0)+i\phi'(x_0)(R_{x_0}e^{i\phi(x_0)}-e^{-i\phi(x_0)})$$
thanks to $(K^0(f))'(x)=-\phi'(x)\int_{x_0}^x M(y)f(y)e^{-2i(\phi(x)-\phi(y))}dy$ and to $v^{<}(x_0)=b(x_0)e^{i\phi(x_0)}, (v^{<})'(0)=b'(x_0)e^{i\phi(x_0)}+i\phi'(x_0)b(x_0)e^{i\phi(x_0)}$.\\ As there is no value of $R_{x_0}$ such that $(S^{>}v^{<})(x_0)=(S^{>}v^{<})'(x_0)=0$,  $S^{>}v^{<}$ is a nonzero element of ${\mathcal U}^{>}_{+ \infty}$.

%
%

Now consider a nonzero element $u$ of ${\mathcal U}^{>}_{+ \infty}$. There exists a sequence of functions $u(x;\sigma), \sigma<0$ solution of (\ref{eq:model-sigma}) with $u(x,\sigma)\rightarrow u(x)$ for $\sigma\rightarrow 0$ and $u(x, \sigma)\rightarrow 0$ when $x\rightarrow +\infty$.

The latter condition implies $u(x,\sigma)$ bounded on $[0, +\infty)$. Thanks to Proposition \ref{teo:Vseries}, one has, under the condition $M_{\infty}^{x_0}<2$
$$u(x;\sigma)=A_{\sigma}(x_0)(1+\sum_{j=1}^{\infty}s^{>}_{\sigma,j}(x))v^{<}(x;\sigma).$$
Let us now show that $\lim_{\sigma \to 0^{-}} A_{\sigma}(x_0)$ exists and is nonzero. As $u$ is nonzero, there exists $X_0>0$ such that $u(X_0)\not=0$.

We had observed that $S^{>}v^{<}$ is nonzero. Hence there exists $X_1>0$ such that $(S^{>}v^{<})(X_1)\not=0$ and there exists $\sigma_*$ such that $\vert (S^{>}_{\sigma}v^{<}_{\sigma})(X_1)\vert>\frac{\vert (S^{>}v^{<})(X_1)\vert }{2}>0$ and $|u(X_1;\sigma)| < 2 |u(X_1)|$ for all $\sigma_*<\sigma<0$. Hence
$$\vert A_{\sigma}(x_0)\vert = \vert \frac{u(X_1,\sigma)}{(S^{>}_{\sigma}v^{<}_{\sigma})(X_1)}\vert \leq \frac{2}{\vert (S^{>}v^{<})(X_1)\vert } 2 |u(X_1)|,
$$
which shows that $|A_{\sigma}(x_0)| \leq C < \infty$. Therefore, there exists a subsequence $\sigma_n \to 0^{-}$ such that $\lim_{n \to \infty} A^{\sigma_n}(x_0)$ exists. This limit cannot be zero, because then $u(x,\sigma)\rightarrow 0$ for each $x$, which is in contradiction with the fact that $u(X_0)\not=0$.\\
We thus conclude that $A_{\sigma}(x_0)$ has a nonzero limit, that we call $A_*$. This proves that $u(x)=A_*S^{>}(x)v^{<}(x)$ for all $x$. The first item of Proposition \ref{teo:S0} is proven. The estimates of the two other items are easy consequences of the inequalities of Proposition \ref{teo:Vseries}. Note that this result is equivalent to the general result for $u$ belonging to ${\mathcal U}^{>}_{+\infty}$  stated  in Corollary  \ref{cod}. It is equivalent to finding solutions $u$ of (\ref{eq:model-problem}) such that $b^{-1}(x)u(x)e^{-i(1+i\sigma)\phi(x)}\rightarrow 0$ when $x\rightarrow +\infty$.

\end{proof}

 We may notice that this expansion is similar to the Bremmer coupling series for the one-way operators, see \cite{3} for example.
 
 \begin{proof}[Proof of Theorem \ref{teo:behavior-at-infinity}]
 Consider $S^>$ given by (\ref{eq:S>}) as before. As $u$ is an outgoing solution, there exists a constant $A_*$ such that $u=A_*S^{>}v^{<}$. The functions $A$ and $B$ introduced in the proof of Proposition \ref{teo:S0} above satisfy $u=A_*(Av^{>}+Bv^{<}), u'=A_*(A(v^{>})'+B(v^{<})')$. Moreover, $B(x_0)=\frac{1}{2i}\int_{x_0}^{+\infty}M(y)S^{>}(y)dy=R_{x_0}$ and $A(x_0)=1$. This leads to the desired equality
 $$\frac{u'(x_0)}{u(x_0)}=\frac{A(v^{>})'+B(v^{<})'}{Av^{>}+Bv^{<}}(x_0)=\frac{(v^>)'(x_0)+R_{x_0}(v^<)'(x_0)}{v^>(x_0)+R_{x_0}v^<(x_0)}.$$
 Reciprocally, assume that $u$ is a solution of (\ref{eq:model-problem}) satisfying
 $$\frac{u'(x_0)}{u(x_0)}=\frac{(v^>)'(x_0)+R_{x_0}(v^<)'(x_0)}{v^>(x_0)+R_{x_0}v^<(x_0)}.$$
 Introduce 
 $$u_*=\frac{u(x_0)}{S^{>}(x_0)v^{>}(x_0)}S^{>}v^{>}.$$
 As $S^{>}v^{>}$ is an outgoing solution at $+\infty$, so is $u_*$. Hence $\frac{u'_*(x_0)}{u_*(x_0)}=\frac{(v^>)'(x_0)+R_{x_0}(v^<)'(x_0)}{v^>(x_0)+R_{x_0}v^<(x_0)}$. Using $u_*(x_0)=u(x_0)$, one deduces $u'_*(x_0)=u'(x_0)$. As $u_*$ and $u$ are solutions of (\ref{eq:model-problem}), by uniqueness $u=u_*$ and $u$ is outgoing at $+\infty$. We have the characterization of $u$ outgoing at $+\infty$ through the condition at $x_0$ of Theorem \ref{teo:behavior-at-infinity}.
  \end{proof}
  
\section{Expression of the reflection coefficient}
\label{sec-limit-R}

\subsection{General formula for $\alpha > 0$}

Introduce $u\in C^1([0, +\infty[)$ a solution of (\ref{eq:model-problem}), and define $a^{>},a^{<}$ the functions such that
\begin{equation}
\label{u-0}
\left\{\begin{array}{l}u(x)=a^{>}(x)w^>(x)+a^{<}(x)w^<(x), x\in [0,x_0[\cr
(a^{>})'(x)w^>(x)+(a^{<})'(x)w^<(x)=0, x\in [0,x_0[
\end{array}\right.
\end{equation}
Introduce

\begin{equation}
\label{value:k}\begin{array}{ll}k\!\!&=e^{-2i\phi(x_0)}\frac{1-ibb'(x_0)+\frac{R_{x_0}e^{i\phi(x_0)}-e^{-i\phi(x_0)}}{R_{x_0}e^{i\phi(x_0)}+e^{-i\phi(x_0)}}}{1+ibb'(x_0)-\frac{R_{x_0}e^{i\phi(x_0)}-e^{-i\phi(x_0)}}{R_{x_0}e^{i\phi(x_0)}+e^{-i\phi(x_0)}}}\cr
&=e^{-2i\phi(x_0)}\frac{(2-ibb'(x_0))R_{x_0}-ibb'(x_0)e^{-2i\phi(x_0)}}{ibb'(x_0)R_{x_0}+(2+ibb'(x_0)e^{-2i\phi(x_0))}}\end{array}
\end{equation}
\begin{proposition}
\label{p:outgoing} A function $u$, solution of (\ref{eq:model-problem}), belongs to $\mathcal{U}_{+\infty}^{>}$ if and only if $a^{<}(x_0)=k \, a^{>}(x_0)$.
\end{proposition}

\begin{proof}
Recall that $u\in \mathcal{U}_{+\infty}^{>}$ if and only if $\frac{u'(x_0)}{u(x_0)}=i\frac{R_{x_0}-1}{R_{x_0}+1}$. As all continuous on $[0, +\infty[$ solutions of (\ref{eq:model-problem}) belong to $C^1$, thanks to $\alpha>0$, equalities (\ref{u-0}) yield
$$\frac{u'(x_0)}{u(x_0)}=-i\phi'(x_0)\frac{1-ke^{2i\phi(x_0)}}{1+ke^{2i\phi(x_0)}}=i\frac{R_{x_0}-1}{R_{x_0}+1},$$
hence the relation $a^{<}(x_0)=k \, a^{>}(x_0)$. 
\end{proof}
%
%
%

The behavior of $R_{x_0}$, which is one of the key points of the estimate of the reflection coefficient $R$ , is given by the following proposition.
\begin{proposition}
\label{p:Rx0-principal}
The limit of $\frac{R_{x_0}}{\theta}$ when $\theta\rightarrow 0$ is,
for $\alpha\geq 1$,
$$\frac{1}{2i}[e^{-2ix_0}\sum_{p=2}^{[\alpha]+1}\frac{d^p}{dy^p}(\frac{y^{\alpha}}{(2i)^{p+1}})+\int_{x_0}^{+\infty}\frac{e^{-2iy}}{(2i)^{[\alpha]+2}}\frac{d^{[\alpha]+2}}{dy^{[\alpha]+2}}(y^{\alpha})dy],$$
and, for $0<\alpha<1$,
$$\frac{\alpha(\alpha -1)}{(2i)^3}\int_{x_0}^{+\infty}y^{\alpha-2}e^{-2iy}dy.$$
\end{proposition}
The proof of Proposition \ref{p:Rx0-principal} is the content of Section \ref{proof:Rx_0-principal}.

We now have all the ingredients to complete the proof of Theorem \ref{teo:main} with the rough estimate $o(\theta)$ for the remainder.

As $R_{x_0}=O(\theta)$ (from Proposition \ref{p:Rx0-principal}), it is straightforward, using $b'(x_0)=O(\theta)$ and $\phi(x_0)=1+O(\theta)$,  and the expression of $k$ given by (\ref{value:k}), to deduce
\begin{equation}
\label{estimate-k}
k=R_{x_0}+\frac{bb'(x_0)}{2i}+O(\theta^2).
\end{equation}
With this expression for $k$, the solution of the Cauchy problem on $[0, x_0[$ with Cauchy data at $x=x_0$, yields the leading-order behavior of the coefficient $R$.
\begin{proposition}
\label{Cauchy-0}
\begin{enumerate}
\item If $u$ is solution of (\ref{eq:model-problem}), then  $\left(\ba{c}a^{>}\cr
a^{<}\ea\right)$ defined through (\ref{u-0}) solves
$$\left(\ba{c}a^{>}\cr
a^{<}\ea\right)(y)=\left(\ba{c}a^{>}\cr
a^{<}\ea\right)(x_0)+\mathcal{V}(\left(\ba{c}a^{>}\cr
a^{<}\ea\right))(y), \forall y \in [0, x_0]$$
the matricial Volterra operator $\mathcal{V}$ being defined by (\ref{Volterra-V}).
\item There exists $\theta_1>0$ such that, for all $\theta<\theta_1$, one has the equality
$$R=\frac{[(Id-\mathcal{V})^{-1}(0)\left(\ba{c}1\cr
k\ea\right)]_{2}}{[(Id-\mathcal{V})^{-1}(0)\left(\ba{c}1\cr
k\ea\right)]_1}.$$
(where $[ \cdot ]_j$ denotes the $j$-th component.)
\item  For $\theta<\mbox{min}(\theta_0(x_0),\theta_1)$, one obtains
\begin{equation}
R=R_{x_0}+\frac{bb'(x_0)}{2i}e^{-2i\phi(x_0)}+\int_0^{x_0}\frac{b'}{b}(y)e^{-2i\phi(y)}dy+O(\theta^2).
\label{R:equality}
\end{equation}
\end{enumerate}

\end{proposition}
\begin{proof}
Item 1):

Consider $u$ given by (\ref{u-0}), with $(a^{>})'e^{-i\phi(x)}+(a^{<})'e^{i\phi(x)}=0$. Equation (\ref{eq:model-problem}) yields
$$\left(\ba{c}a^{>}\cr
a^{<}\ea\right)'= \frac{\phi''(x)}{2\phi'(x)}\left(\ba{cc}i\phi'(x)e^{i\phi(x)}&-e^{i\phi(x)}\cr
i\phi'(x)e^{-i\phi(x)}&e^{-i\phi(x)}\ea\right)\left(\ba{cc}0&0\cr
e^{-i\phi(x)}&-e^{i\phi(x)}\ea\right)\left(\ba{c}a^{>}\cr
a^{<}\ea\right),$$
which implies
\begin{equation}\label{Volterra-V}\begin{array}{ll} \left(\ba{c}a^{>}\cr
a^{<}\ea\right)(x)&=\left(\ba{c}a^{>}\cr
a^{<}\ea\right)(x_0)+\int_{x_0}^x\frac{\phi''(y)}{2\phi'(y)}\left(\ba{cc}-1&e^{2i\phi(y)}\cr
e^{-2i\phi(y)}&-1\ea\right)\left(\ba{c}a^{>}\cr
a^{<}\ea\right)(y)dy\cr
&:=\left(\ba{c}a^{>}\cr
a^{<}\ea\right)(x_0)+\mathcal{V}(\left(\ba{c}a^{>}\cr
a^{<}\ea\right))(x),\end{array}\end{equation}
this equality defining the matricial Volterra operator $\mathcal{V}$. 

Item 2) As $\frac{\phi''(x)}{\phi'(x)}\in L^1([0, x_0])$, the kernel of this integral operator belongs to $L^1([0, x_0])$. Observe that, for all $x\in [0, x_0]$, 
$$\vert \mathcal{V}(f)(x)\vert \leq \int_{x}^{x_0}\vert \frac{\phi''(y)}{\phi'(y)}\vert dy \vert \vert f\vert\vert_{L^{\infty}[0, x_0]}$$
As $\phi'=c^{-1}$ and $c$ is decreasing, $\phi'$ is increasing and positive, hence 
$$\vert \mathcal{V}(f)(x)\vert \leq \ln\frac{\phi'(x_0)}{\phi'(x)} \vert \vert f\vert\vert_{L^{\infty}[0, x_0]}.$$
There exist a constant $C(x_0)$, independent on $\theta$, such that
$$\ln \phi'(x_0)\leq C(x_0)\theta$$
The map $U\rightarrow S+\mathcal{V}(U)$ is thus a contraction for $\theta<\theta_1$, $\theta_1$ small enough, hence has a unique fixed point, which yields the unique solution $ \left(\ba{c}a^{>}\cr
a^{<}\ea\right)$ of (\ref{Volterra-V}) which is $$ \left(\ba{c}a^{>}\cr
a^{<}\ea\right)(x)=(Id-\mathcal{V})^{-1}(y)\left(\ba{c}a^{>}\cr
a^{<}\ea\right)(x_0).$$
Proposition \ref{p:outgoing} yields $\left(\ba{c}a^{>}\cr
a^{<}\ea\right)(x_0)=a^{>}(x_0)\left(\ba{c}1\cr
k\ea\right)$. 

Let $r(x)=\left(\ba{c}a^{>}\cr
a^{<}\ea\right)(x)-a^{>}(x_0)[\left(\ba{c}1\cr
k\ea\right)\mathcal{V}(\left(\ba{c}1\cr
k\ea\right))(x)].$ It satisfies
$$r(x)=a^{>}(x_0)\mathcal{V}^2(\left(\ba{c}1\cr
k\ea\right))(x)+\mathcal{V}(r)(x).$$
As $\mathcal{V}$ is a contraction for $\theta<\theta_1$, here exists $c$ such that $\vert \vert r\vert\vert_{(L^{\infty}[0, x_0])^2}\leq c \, \theta^2$.

From (\ref{coeffAB}) and the relations $u(0)=a^>(0)+a^{<}(0), u'(0)=ia^{>}(0)R-ia^{>}(0)$, one deduces that $R=\frac{a^{<}}{a^{>}}(0)$. From $$ \left(\ba{c}a^{>}\cr
a^{<}\ea\right)(0)=(Id-\mathcal{V})^{-1}(0)\left(\ba{c}a^{>}\cr
a^{<}\ea\right)(x_0)=a^{>}(x_0)(Id-\mathcal{V})^{-1}(0)\left(\ba{c}1\cr
k\ea\right)(x_0),$$
one obtains the reflection coefficient $R$ thanks to $a^{>}(0)\not=0$, hence proving item 2) of Proposition \ref{Cauchy-0}. 

Item 3)
Use $k=O(\theta)$, one obtains
$$\left(\ba{c}a^{>}\cr
a^{<}\ea\right)(0)=a^{>}(x_0)[\left(\ba{c}1\cr
k\ea\right)+\mathcal{V}(\left(\ba{c}1\cr
0\ea\right))(x)]+O(\theta^2).$$
As
$$\mathcal{V}(\left(\ba{c}1\cr
0\ea\right))(0)=\int_{x_0}^0\frac{\phi''(x)}{2\phi'(x)}\left(\ba{cc}-1&e^{2i\phi(x)}\cr
e^{-2i\phi(x)}&-1\ea\right)\left(\ba{c}1\cr
0\ea\right)dx=\left(\ba{c}-\int_{x_0}^0\frac{\phi''(x)}{2\phi'(x)}dx\cr
\int_{x_0}^0\frac{\phi''(x)}{2\phi'(x)}e^{-2i\phi(x)}dx\ea\right),$$
using $b^2\phi'=1$ and $\frac{\phi''}{\phi'}=O(\theta)$, one obtains:
$$\mathcal{V}(\left(\ba{c}1\cr
0\ea\right))(0)=\left(\ba{c}O(\theta)\cr
\int_{x_0}^0\frac{\phi''(x)}{2\phi'(x)}e^{-2i\phi(x)}dx\ea\right)=\left(\ba{c}O(\theta)\cr
\int^{x_0}_0\frac{b'(x)}{b(x)}e^{-2i\phi(x)}dx\ea\right).$$
The first component of $[\left(\ba{c}1\cr
k\ea\right)+\mathcal{V}(\left(\ba{c}1\cr
0\ea\right))(0)]$ is $1+O(\theta)$ while the second component is $k+\int^{x_0}_0\frac{b'(x)}{b(x)}e^{-2i\phi(x)}dx$,
from which one deduces
$$R=k+\int^{x_0}_0\frac{b'}{b}(x)e^{-2i\phi(x)}dx+O(\theta^2)=R_{x_0}+\frac{bb'(x_0)}{2i}e^{-2i\phi(x_0)}+\int^{x_0}_0\frac{b'}{b}(x)e^{-2i\phi(x)}dx+O(\theta^2).$$
\end{proof}
\paragraph{Proof of Theorem \ref{teo:main}}
From (\ref{R:equality}), one deduces
$$\frac{R}{\theta}=\frac{R_{x_0}}{\theta}+\frac{\alpha x_0^{\alpha-1}}{(2i)^3}e^{-2ix_0}-\int_0^{x_0}\frac{\alpha y^{\alpha-1}}{4}e^{-2iy}dy+O(\theta)$$
From Proposition \ref{p:Rx0-principal}, one has
$$\frac{R}{\theta}=\frac{1}{2i}[e^{-2ix_0}\sum_{p=2}^{[\alpha]+1}\frac{d^p}{dy^p}(\frac{y^{\alpha}}{(2i)^{p+1}})+\int_{x_0}^{+\infty}\frac{e^{-2iy}}{(2i)^{[\alpha]+2}}\frac{d^{[\alpha]+2}}{dy^{[\alpha]+2}}(y^{\alpha})dy]-\frac{\alpha x_0^{\alpha-1}}{8i}+\frac{1}{(2i)^2}\int_0^{x_0}\frac{d}{dy} y^{\alpha}e^{-2iy}dy+o(1)$$
Using equality (\ref{eq:Gamma}), the above expression can be written as
$$\frac{R}{\theta}=\frac{\Gamma(\alpha+1)}{(2i)^{\alpha+2}}+o(1),$$
which proves Theorem \ref{teo:main}.

When $\alpha>1$, remark that $\frac{bb'(x_0)}{2i}+\int_0^{x_0}\frac{b'}{b}(x)e^{-2i\phi(x)}dx=\frac{1}{2i}\int_0^{x_0}(M(y)+(b'(y))^2)e^{-2i\phi(y)}dy$. This allows to verify that, in the case $\alpha>1$, one could have chosen $x_0=0$.

\subsection{Calculation of $\lim_{\theta\rightarrow 0}\frac{R_{x_0}}{\theta}$}
\label{proof:Rx_0-principal}

We now address the proof of Proposition \ref{p:Rx0-principal}.

One relies in this section on a fixed value of $x_0>0$ independent on $\theta$. In the case $0<\alpha<1$, the calculation of the limit, as well as the estimate of the remainder term $\frac{R_{x_0}}{\theta}-\lim_{\theta\rightarrow 0}\frac{R_{x_0}}{\theta}$, are both straightforward, as the following lemma shows.
\begin{lemma}\label{estimate-remainder1}
For $0<\alpha<1$, there exists $\theta_0(x_0)$ and $c$ such that, for all $\theta<\theta_0(x_0)$ and $\theta<c^{-1}$
$$\vert R_{x_0}-\frac{1}{2i}\int_{x_0}^{+\infty}M(y)e^{-2i\phi(y)}dy\vert\leq \frac{c^2\theta^2}{1-c\theta}.$$
and one gets the estimate
$$\vert\frac{R_{x_0}}{\theta}+\frac{1}{2i}\int_{x_0}^{+\infty}y^{\alpha-2}T_0(0)dy\vert\leq C\theta.$$
 \end{lemma}
Indeed, there exists a constant $C$ and a constant $c$ such that, for $x_0>0$ given and $0<\alpha<1$

\begin{equation} M_{x_0}\leq \theta \frac{x_0^{\alpha-1}}{1-\alpha}C=2c\theta
\end{equation}
Under this estimate, one checks that, for $0<\alpha<1$,
\begin{equation}
\vert S^{>}(x)-e^{-2i\phi(x)}\vert\leq  \frac{c\theta}{1-c\theta},
\end{equation}
and one deduces that
$$\vert R_{x_0}-\frac{1}{2i}\int_{x_0}^{+\infty}M(x')e^{-2i\phi(x')}dx'\vert\leq (c\theta)^2\frac{1}{1-c\theta }.$$
This proves Lemma \ref{estimate-remainder1}.

The proof for $\alpha>1$ is more complicated. Indeed, one of the difficulties is that, when $\theta>0$, $\theta^{-1}M(x)$ belongs to $L^1([x_0, +\infty[)$ while its limit when $\theta\rightarrow 0$ is $-\frac{\alpha(\alpha-1)}{4}x^{\alpha-2}$, which does not belong to $L^1([x_0, +\infty[)$.

As $M$ is multiplied by an oscillating term, we will use repeated integrations by parts to decrease the degree of $x^{\alpha-2}$.

Before writing the details of the proof of Proposition \ref{p:Rx0-principal} in this case, let us give a synopsis of it. 

Introduce the operator $\mathcal{L}$ given by 

\begin{equation}\label{mathcalL}\mathcal{L}(f)(x)=\frac{d}{dx}(c f)(x),\end{equation}

where $c(x)$ is from equation (\ref{c-fractional}).

\begin{itemize}
\item Using the estimate $M_{x_0}=O(\theta^{\frac{1}{\alpha}})$, it is enough to study, instead of $R_{x_0}$, the quantity 
$$\frac{1}{2i}\int_{x_0}^{+\infty}M(y)e^{-2i\phi(y)}(\sum_{j=0}^{j_0-1}s^{>}_{0,j}(y))dy$$
with $\frac{j_0+1}{\alpha}>1> \frac{j_0}{\alpha}$,

\item Using repetitive integration by parts (Lemma \ref{estimate-integration-by-parts2}), one gets the identity
\begin{equation}
\int_{x_0}^{+\infty}f(y)e^{-2i\phi(y)}dy=\frac{c(x_0)}{2i}e^{-2i\phi(x_0)}[\sum_{p=0}^{n-1}(2i)^{-p}\mathcal{L}^p(f)(x_0)]+(2i)^{-n}\int_{x_0}^{+\infty}\mathcal{L}^n(f)(y)e^{-2i\phi(y)}dy,
\label{eq:formula-ipp}
\end{equation}
\item Observe finally that, for $f(y)= T_0(\theta y^{\alpha})(\sum_{j=0}^{j_0-1}s^{>}_{0,j}(y))$, $\mathcal{L}^p(f)(x_0)$ converges, when $\theta\rightarrow 0$, to $\frac{d^p}{dx^p}(T_0(0) x^{\alpha-2})(x_0)$. \end{itemize}

These steps being proven, we choose, for $\alpha\notin \N$, $n$ such that $\alpha-2-n<-1$. The dominated convergence theorem proves that $\int_{x_0}^{+\infty}\mathcal{L}^n(f)(y)e^{-2i\phi(y)}dy$ converges, when $\theta\rightarrow 0$, to $\int_{x_0}^{+\infty}\frac{d^n}{dy^n}(T_0(0) y^{\alpha-2})e^{-2iy}dy$. 

Remark that a special treatment shall be used for $\alpha\in \N$.

The proof of Proposition \ref{p:Rx0-principal} relies on Lemma \ref{estimate-integration-by-parts2} and Appendix \ref{p:all-terms}. We introduce $j_0$ such that $j_0< \alpha<j_0+1$. 

%
The first Lemma suppresses all terms whose $L^{\infty}$ norm is a $o(\theta)$, for all $\alpha>0$:
\begin{lemma}\label{estimate-remainder2}
Let $\alpha>1$. For all $\theta$ such that $\theta<\theta_0(x_0)$ (see Lemma \ref{ineq:Mx0}), one has
$$\vert R_{x_0}-\frac{1}{2i}\int_{x_0}^{+\infty}M(y)e^{-2i\phi(y)}(\sum_{j=0}^{j_0-1}s^{>}_{0,j}(y))dy\vert\leq \theta^{\frac{j_0+1}{\alpha}}(\frac{I}{2})^{j_0+1} (1-\frac12 \theta^{\frac{1}{\alpha}}I)^{-1}.$$
\end{lemma}
The second lemma yields a formula of integration by parts which is crucial for the estimate on $[x_0, +\infty[$ based on 
$$\int_{x_0}^{+\infty}f(y)e^{-2i\phi(y)}dy=\frac{f(x_0)c(x_0)}{2i}e^{-2i\phi(x_0)}+\frac{1}{2i}\int_{x_0}^{+\infty}\mathcal{L}(f)(y)e^{-2i\phi(y)}dy,$$
for $f\in L^1([x_0, +\infty[)$ and $\mathcal{L}(f)\in L^1([x_0, +\infty[)$. It can be generalized into 


\begin{lemma}
\label{estimate-integration-by-parts2}
Let $f$ be a function of class $C^{j_0-1}([0, +\infty[)$, bounded, such that all derivatives are bounded by $y^{-1}$ when $y$ large. For every $n_0\geq 0$, (\ref{eq:formula-ipp}) holds.
\end{lemma}
\paragraph{Proof of Proposition \ref{p:Rx0-principal}}
We differentiate two cases for $\alpha$. Indeed, the proof we rely on is based on the behavior of $\mathcal{L}^n(f)$ for $f$ proportional to $x^{\alpha-2}$, and one of the arguments that may be used is the fact that $\mathcal{L}^n(x^{\alpha-2})$ behaves as $x^{\alpha-2-n}$, and that $\alpha-2-n$ is strictly smaller than $-1$ for $n$ large enough, except if $\alpha$ is integer because the derivative of $x^0$ is 0.
\paragraph{First case: $\alpha\notin \N$}
Let $f$ be given by \begin{equation}f(y)= y^{\alpha-2}T_0(\theta y^{\alpha})(\sum_{j=0}^{j_0-1}s^{>}_{0,j}(y)).\label{choix-de-f}\end{equation}
such that $M(y)(\sum_{j=0}^{j_0-1}s^{>}_{0,j}(y))=-\theta f(y)$.
There exists $\theta_0(x_0)$ and $D_0$ such that, for $\theta<\theta_0(x_0)$:
$$\forall y\geq x_0, \vert\vert f(y)\vert\vert\leq D_0y^{\alpha-2}.$$
This comes from $T_0$ bounded, $M_{x_0}$ bounded by $K_0\theta^{\frac{1}{\alpha}}$, and all the terms of the expansion are bounded (Proposition \ref{teo:S0}).

Thanks to $c(y)\rightarrow 1$ for all $y\in [x_0, +\infty[$ when $\theta\rightarrow 0$,  $f(y)$ converges to $y^{\alpha-2}T_0(0)$ when $\theta\rightarrow 0$, that is $\mathcal{L}(f)$ converges to $(\alpha-2)y^{\alpha-3}T_0(0)$.

Moreover, uniformly on $[x_0, +\infty[$, for all $n\geq 0$;
\begin{equation}\mbox{lim}_{\theta\rightarrow 0}\mathcal{L}^n(f)(x)= \frac{d^n}{dy^n}(y^{\alpha-2})T_0(0).\label{limitn}\end{equation}
Let us choose $n_0=[\alpha]$ and recall that $x_0$ is given, independent on $\theta$. As $\alpha\notin \N$, for all $p\geq 1$, there exists $T_p$, bounded on $[0, +\infty)$, such that 
$$\mathcal{L}^{p}(f)(x)= x^{\alpha-2-p}T_{p}(\theta x^{\alpha}),$$
with $T_p(0)\not=0$ for all $p\geq 0$ thanks to $\alpha\notin \N$.

The estimates obtained in Lemma \ref{l:A2} allow us to use the dominated convergence theorem on $\int_{x_0}^{+\infty}\mathcal{L}^{n_0}(M) e^{-2i\phi(y)}dy$. This method can be applied in the two cases studied in the present paper (one gets $\mathcal{L}^n(M)(x)=\theta x^{\alpha-2-n}T_n^*(\theta x^{\alpha})$, where $T_n^*$ is a bounded on $[0, +\infty[$ rational fraction in the model case (\ref{eq:c_lambda}) and it is a compactly supported function for (\ref{eq:c_lambdaofx}), hence in both cases it is uniformly bounded by a constant $K_0$) and the use of the limit (\ref{limitn}) allows us to obtain:
$$
\begin{array}{ll}\mbox{lim}_{\theta\rightarrow 0}\int_{x_0}^{+\infty}f(y)e^{-2i\phi(y)}dy&=\frac{e^{-2ix_0}}{2i}\sum_{p=0}^{n_0-1}\mbox{lim}_{\theta\rightarrow 0}\frac{\mathcal{L}^p(f)}{(2i)^{p}}(x_0))+\int_{x_0}^{+\infty}\mbox{lim}_{\theta\rightarrow 0}\frac{\mathcal{L}^{n_0}(f)}{(2i)^{n_0}}(y)e^{-2iy}dy\cr
&=T(0)[\frac{e^{-2ix_0}}{2i}\sum_{p=0}^{[\alpha]-1}\frac{1}{(2i)^{p}}\frac{d^p}{dy^p}(y^{\alpha-2})(x_0))+\frac{1}{(2i)^{[\alpha]}}\int_{x_0}^{+\infty}\frac{d^{[\alpha]}}{dy^{[\alpha]}}(y^{\alpha-2})e^{-2iy}dy].\end{array}
$$
that is, using $T(0)=-\frac{\alpha(\alpha-1)}{(2i)^2}$ and $T_{[\alpha]}(0)=(\alpha-2)...(\alpha-[\alpha]-1)T_0(0)$
\begin{equation}
\begin{array}{ll}-\mbox{lim}_{\theta\rightarrow 0}\int_{x_0}^{+\infty}f(y)e^{-2i\phi(y)}dy&=\frac{e^{-2ix_0}}{2i}[\sum_{p=0}^{n_0-1}\mbox{lim}_{\theta\rightarrow 0}\frac{\mathcal{L}^p(f)}{(2i)^{p}}(x_0)]+\int_{x_0}^{+\infty}\mbox{lim}_{\theta\rightarrow 0}\frac{\mathcal{L}^{n_0}(f)}{(2i)^{n_0}}(y)e^{-2iy}dy\cr
&=\frac{e^{-2ix_0}}{2i}[\sum_{p=2}^{[\alpha]+1}\frac{1}{(2i)^{p}}\frac{d^p}{dy^p}(y^{\alpha})(x_0))]+\frac{1}{(2i)^{[\alpha]+2}}\int_{x_0}^{+\infty}\frac{d^{[\alpha]+2}}{dy^{[\alpha]+2}}(y^{\alpha})e^{-2iy}dy].\end{array}
\end{equation}
One recognizes the last part of the expression (\ref{eq:Gamma}). 
\paragraph{Second case: $\alpha\in \N$, denoted by $n$.}
Note that there exists, for $p\geq 1$, a function $T_p\in L^1([0, +\infty[)$ (it is a bounded on $[0, +\infty[$ rational fraction in the model case (\ref{eq:c_lambda}) and it is a compactly supported function for (\ref{eq:c_lambdaofx})) such that $\mathcal{L}^{n_0}(f)(x)= x^{n-2-n_0}T_{n_0}(\theta x^n)$, and $T_{n_0}(0)=(n-2)(n-3)...(n-1-n_0)T_0(0)$ for every $n_0$. The function $x\rightarrow T_{n_0}(\theta x^n)$ is uniformly bounded by a constant $K_{n_0}$. Choose $n_0=n-2$. One has
$$\mathcal{L}^{n-2}(f)(x)=T_{n-2}(\theta x^n).$$
Considering two additional derivatives, there exists a function $Q_n$, uniformly bounded by a constant $C_n$, independent on $\theta$ for $\theta<\theta_0(x_0)$, such that
\begin{equation}\mathcal{L}^n(f)(x)= \theta x^{n-2}[(n-1)c(cT_{n-2})'(\theta x^n)+n\theta x^n (c(cT_{n-2})')'(\theta x^n)]=x^{-2}Q_n(\theta x^n),\label{fstar}\end{equation}
In addition, $Q_n(0)=0$. 
One has thus

$$\int_{x_0}^{+\infty}T_{n-2}(\theta x^n)e^{-2i\phi(x)}dx=c(x_0)e^{-2i\phi(x_0)}[\frac{T_{n-2}(\theta x_0^n)}{2i}+\frac{\mathcal{L}(T_{n-2})(\theta x_0^n)}{(2i)^2}]+\int_{x_0}^{+\infty}x^{-2}\frac{Q_n(\theta x^n)}{(2i)^2}e^{-2i\phi(x)}dx,$$
The uniform bound in the integral by $Cx^{-2}$, $C$ independent on $\theta$ for $\theta<\theta_1$, allows to use the dominated convergence theorem, and one obtains, using $x^{-2}Q_n(\theta x^n)\rightarrow 0$ when $\theta\rightarrow 0$ thanks to $Q_n(0)=0$ (and if $\alpha\in \N$, $\frac{d^{[\alpha]+2}}{dy^{[\alpha]+2}}(y^{\alpha})=0$)

\begin{equation}
\mbox{lim}_{\theta\rightarrow 0}\int_{x_0}^{+\infty}f(y)e^{-2i\phi(y)}dy=T_0(0)e^{-2ix_0}[\sum_{p=0}^{n-2}(2i)^{-p}\frac{d^p}{dy^p}(y^{n-2})]\vert_{y=x_0},\end{equation}
which proves Proposition \ref{p:Rx0-principal} in the case $\alpha\in \N$.

This ends the proof of Proposition \ref{p:Rx0-principal} for all $\alpha>0$.

In the sequel, we shall denote by $\theta_1$ any constant satisfying $0<\theta_1<\theta_0(x_0)$ and by $C$, when needed, a general constant. The aim of the next Section, which is rather technical, is to evaluate the remainder terms in $R_{x_0}$ and its order of magnitude in $\theta$.

\section{Asymptotic bounds for $R$}
\label{sec-limit-theta}
This section quantifies the remainder term in the expression of $R$, and in particular identifies its magnitude in $\theta$. 

The case $0 < \alpha < 1$ was already addressed in Lemma \ref{estimate-remainder1}. It therefore remains to show

\begin{theorem}
\label{th:refl:coeff}
For all $\alpha \geq 1$, there exists a constant $C$ and $\theta_1$ such that, for all $\theta<\theta_1$, 
$$\vert \frac{R}{\theta}-\frac{\Gamma(\alpha+1)}{(2i)^{\alpha+2}}\vert \leq C\theta^{\frac{1}{\alpha}}.$$
\end{theorem}
This is the most technical part of the paper, and we shall differentiate two cases, namely the case where $\alpha$ is integer and the case where $\alpha$ is not an integer. This is a consequence of the method used in this proof, namely the repeated integration by parts we perform for the study of an integral of the form $\int_{x_0}^{+\infty} x^{\alpha-2}T(\theta x^{\alpha})e^{-2i\phi(x)}dx$. In the case $\alpha$ not being an integer, there exists $n$ such that $\alpha-n< -1$ and the coefficient of $x^{\alpha-n}$ in the integral obtained after $n-2$ integration by parts is not zero. In the case $\alpha$ being an integer, if $n=\alpha$ one obtains an integral of the form $\int_{x_0}^{+\infty} T_{n-2}(\theta x^n)e^{-2i\phi(x)}dx$, and an additional integration by parts would lead to an integral containing $x^{n-1}$, which is not in $L^1([1, +\infty))$.

 
 Remark first that an estimate similar to the estimate of Lemma \ref{estimate-remainder2}  is
 $$\vert R_{x_0}-\frac{1}{2i}\int_{x_0}^{+\infty}M(y)e^{-2i\phi(y)}(\sum_{j=0}^{j_0+1}s^{>}_{0,j}(y))dy\vert\leq \theta^{\frac{j_0+2}{\alpha}}(\frac{I}{2})^{j_0+2} (1-\frac12 \theta^{\frac{1}{\alpha}}I)^{-1},$$
 from which one deduces
 \begin{equation}\label{in:alpha}\vert\frac{R_{x_0}}{\theta}-\frac{1}{2i}\int_{x_0}^{+\infty}\theta^{-1}M(y)e^{-2i\phi(y)}(\sum_0^{[\alpha]+1}s^{>}_{j}(y))dy\vert\leq C\theta^{\frac{1}{\alpha}}.\end{equation}
 Note that we are not limited to this estimate, we have as well
 \begin{equation}\label{in:2alpha}\vert\frac{R_{x_0}}{\theta}-\frac{1}{2i}\int_{x_0}^{+\infty}\theta^{-1}M(y)e^{-2i\phi(y)}(\sum_0^{2[\alpha]+1}s^{>}_{j}(y))dy\vert\leq C\theta.\end{equation}

 Appendix \ref{p:all-terms} leads to
 $$\vert\frac{R_{x_0}}{\theta}-\frac{1}{2i}\int_{x_0}^{+\infty}\theta^{-1}M(y)e^{-2i\phi(y)}dy\vert\leq C_2\theta^{\frac{1}{\alpha}}.$$

The proof of Theorem \ref{th:refl:coeff} relies on the estimate
$$\vert \frac{R_{x_0}}{\theta}-\frac{1}{2i}[e^{-2ix_0}\sum_{p=2}^{[\alpha]+1}\frac{d^p}{dy^p}(\frac{y^{\alpha}}{(2i)^{p+1}})+\int_{x_0}^{+\infty}\frac{e^{-2iy}}{(2i)^{[\alpha]+2}}\frac{d^{[\alpha]+2}}{dy^{[\alpha]+2}}(y^{\alpha})dy]\vert\leq C\theta^{\frac{1}{\alpha}}$$
for $\alpha\geq 1$.
As for the limit $\frac{R_{x_0}}{\theta}$, it is split into the cases $\alpha\notin\N$ and $\alpha\in \N$. Observe that $c(x_0)=1+O(\theta)$ and $\phi(x_0)=x_0+O(\theta)$, hence the contributions of order $\theta^{\frac{1}{\alpha}}$ can only come from the terms $s^{>}_{j}$, and from the integrals $\int_{x_0}^{+\infty}\mathcal{L}^p(f)(x)e^{-2i\phi(x)}dx.$

$\bullet$ In the case $\alpha\in \N$, using (\ref{fstar}) one obtains
$$\mathcal{L}^{n+1}(f)=x^{-2}\theta x^{n-1}[(n-2)R_n(\theta x^n)+n\theta x^nR'_n(\theta x^n)].$$ 
Observe that
$$ \int_{x_0}^{+\infty}T_{n-2}(\theta x^{n})e^{-2i\phi(x)}dx=\frac{c(x_0)e^{-2i\phi(x_0)}}{2i}[T_{n-2}+\frac{\mathcal{L}(T_{n-2})}{(2i)}+\frac{\mathcal{L}^2(T_{n-2}}{(2i)^2}])(\theta x_0^n)+\frac{1}{(2i)^3}\int_{x_0}^{+\infty}\mathcal{L}^{n+1}(f)e^{-2i\phi(x)}dx,$$
One proved the uniform optimal bound on $[x_0, +\infty[$ (Lemma \ref{l:A2})
$$\vert \theta x^{n-1}[(n-2)R_n(\theta x^n)+n\theta x^nR'_n(\theta x^n)]\vert\leq C\theta^{\frac{1}{n}}$$
from which one deduces that there exists a constant $C$ such that 
$$\vert \int_{x_0}^{+\infty}T_{n-2}(\theta x^n)e^{-2i\phi(x)}dx-\frac{c(x_0)e^{-2i\phi(x_0)}}{2i}[T_{n-2}+\frac{\mathcal{L}(T_{n-2})}{(2i)}+\frac{\mathcal{L}^2(T_{n-2})}{(2i)^2}](\theta x_0^n)\vert\leq C\theta^{\frac{1}{n}}x_0^{-1}$$
hence, as  $\frac{c(x_0)}{2i}[T_{n-2}(\theta x_0^n)+\frac{\mathcal{L}(T_{n-2})(\theta x_0^n)}{(2i)}+\frac{\mathcal{L}^2(T_{n-2})(\theta x_0^n)}{(2i)^2}]$ depend only on $\theta$ and $\theta\leq (\theta_0(x_0))^{1-\frac{1}{n}}\theta^{\frac{1}{n}}$ for $\theta<\theta_0(x_0)$, there exists a constant $M$ such that, for $\theta<\theta_0(x_0)$
\begin{equation}
\label{estimate-0}
\vert \int_{x_0}^{+\infty}T_{n-2}(\theta x^n)e^{-2i\phi(x)}dx-\frac{T_{n-2}(0)}{2i}e^{-2ix_0}\vert \leq M\theta^{\frac{1}{n}}.\end{equation}

One thus deduces that, using Lemma \ref{estimate-integration-by-parts2} for $n\geq 3$ $$\vert \frac{R_{x_0}}{\theta}-\frac{c(x_0)}{2i}e^{-2i\phi(x_0)}[\sum_{p=0}^{n-3}(2i)^{-p}\mathcal{L}^p(f)(x_0)]-\frac{T_{n-2}(0)}{2i}e^{-2ix_0}\vert \leq 2^{-n+2}M\theta^{\frac{1}{n}},$$ 
and using
$$\frac{c(x_0)}{2i}e^{-2i\phi(x_0)}[\sum_{p=0}^{n-3}(2i)^{-p}\mathcal{L}^p(f)(x_0)]=\frac{1}{2i}e^{-2ix_0}[\sum_{p=0}^{n-3}(2i)^{-p}x_0^{n-2-p}T_p(0)]+O(\theta^{\frac{1}{n}}),$$
we obtain the estimate
$$\frac{R_{x_0}}{\theta}=\frac{1}{2i}e^{-2ix_0}[\sum_{p=0}^{n-2}(2i)^{-p}x_0^{n-2-p}T_p(0)]+O(\theta^{\frac{1}{n}}),$$
which is exactly, thanks to Equality \ref{eq:Gamma}

$$\frac{R_{x_0}}{\theta}=\frac{\Gamma(n+2)}{(2i)^{n+2}}-\int_0^{x_0}y^{n+1}e^{-2iy}dy+O(\theta^{\frac{1}{n}}).$$
Theorem \ref{th:refl:coeff} is thus proven in the case $\alpha\in \N$.

$\bullet$ In the case $\alpha \notin\N$, $\alpha>1$:

  Use again Appendix \ref{p:all-terms} with an exact expression of $r_0(x_0,\theta)$. One has the identity (coming from Lemma \ref{l:A1} for $j=0$ and $m=M$):
 $$\int_{x_0}^{+\infty}\theta^{-1}M(y)e^{-2i\phi(y)}dy=\frac{c(x_0)}{2i}T_k^S(x_0^{-1}, \theta x_0^{\alpha})+\frac{1}{(2i)^{k+1}}\int_{x_0}^{+\infty}\mathcal{L}^{k+1}(\theta^{-1}M)e^{-2i\phi(y)}dy.$$
 Notice first that $\vert \frac{c(x_0)}{2i}T_k^S(x_0^{-1}, \theta x_0^{\alpha})-\frac{1}{2i}T_k^S(x_0^{-1},0)\vert=O(\theta)$. One has just to study the second term of this equality. Denote this second term by $J$. We use then the identity (\ref{symbolep}) of Lemma \ref{l:A2}:
 $$-J=\frac{1}{(2i)^{k+1}}\int_{x_0}^{+\infty}y^{\alpha-3-k}T_{k+1}(\theta y^{\alpha})e^{-2i\phi(y)}dy.$$
 We assume now that $\alpha-2-k<0$. The uniform bound $\vert T_{k+1}\vert \leq T_{k+1}^{\infty}$ translates into $\vert y^{\alpha-3-k}T_{k+1}(\theta y^{\alpha})\vert\leq y^{\alpha-3-k}T_{k+1}^{\infty}$, hence the integral converges for all $\theta\geq 0$.\\
   
 Use the change of variable $y=\theta^{-\frac{1}{\alpha}}X$. One obtains
 $$-J=\frac{1}{(2i)^{k+1}}\theta^{\frac{2+k}{\alpha}}\int_{\theta^{\frac{1}{\alpha}}x_0}^{+\infty}X^{\alpha-3-k}T_{k+1}(X^{\alpha})e^{-2i\theta^{-\frac{1}{\alpha}}Xh(X^{\alpha})}dX.$$
 Use the change of variable $u=Xh(X^{\alpha})$. As $X\rightarrow \sqrt{1+X^{\alpha}}$ is greater than 1, $X\rightarrow u(X)$ is a diffeomorphism from $[0, +\infty[$ to $[0, +\infty[$, such that $u'(0)=1$, hence its reciprocal is also a diffeomorphism and $X'(0)=1$. One deduces
 $$J=-\frac{1}{(2i)^{k+1}}\theta^{\frac{2+k}{\alpha}}\int_{\theta^{\frac{1}{\alpha}}x_0h(\theta x_0^\alpha)}^{+\infty}(1+(X(u))^{\alpha})^{-\frac12}(X(u))^{\alpha-3-k}T_{k+1}((X(u))^{\alpha})e^{-2i\theta^{-\frac{1}{\alpha}}u}du.$$
Denote by $Q(u)$ the quantity such that
 $$(1+(X(u))^{\alpha})^{-\frac12}(X(u))^{\alpha-3-k}T_{k+1}((X(u))^{\alpha})-(1+u^{\alpha})^{-\frac12}u^{\alpha-3-k}T_{k+1}(u^{\alpha})=u^{\alpha}Q(u).$$

 When $0\leq u\leq 1$, observing that there exists a function $Z$ solving $Z(\tau)h(\tau (Z(\tau))^{\alpha})=1$ such that $X(u)=uZ(u^{\alpha})$, there exists a function $r$ smooth such that $$u-X(u)=\frac{1}{h(u^{\alpha}(Z(u^{\alpha}))^{\alpha})}[h(u^{\alpha}(Z(u^{\alpha}))^{\alpha})-1]=u^{\alpha}r(u^{\alpha}),$$
 which imply that $Q$ is bounded in the neighborhood of 0.\\
 On the other side, the estimate, when $X$ is large, $u\simeq X^{\frac{\alpha}{2}+1}$ leads to $(1+(X(u))^{\alpha})^{-\frac12}(X(u))^{\alpha-3-k}\simeq u^{-1}u^{-\frac{2+k}{\frac{\alpha}{2}+1}}$. \\
 One has then $$-J=\frac{1}{(2i)^{k+1}}\theta^{\frac{2+k}{\alpha}}\int_{x_0h(\theta x_0^\alpha)}^{+\infty}(1+\theta t^{\alpha})^{-\frac12}t^{1-\frac{3+k}\alpha}T_{k+1}(\theta t^{\alpha})e^{-2it}dt+\frac{1}{(2i)^{k+1}}\theta^{\frac{2+k}{\alpha}}\int_{\theta^{\frac{1}{\alpha}}x_0h(\theta x_0^\alpha)}^{+\infty}u^{\alpha}Q(u)e^{-2i\theta^{-\frac{1}{\alpha}}u}du.$$
 The last term of this equality is equal to
 $$K=\frac{1}{(2i)^{k+1}}\theta^{1+\frac{3+k}{\alpha}}\int_{x_0h(\theta x_0^\alpha)}^{+\infty}u^{\alpha}Q(\theta^{\frac{1}{\alpha}}t)e^{-2it}dt$$
 The estimates of $(1+u^{\alpha})^{-\frac12}u^{1-\frac{3+k}{\alpha}}T_{k+1}(u^{\alpha})$ and of $(1+(X(u))^{\alpha})^{-\frac12}(X(u))^{\alpha-3-k}T_{k+1}((X(u))^{\alpha})$ at $+\infty$ show that $K=O(\theta^{1+\frac{3+k}{\alpha}})$.
One then deduces

$$-\theta^{-1}J=\frac{1}{(2i)^{k+1}}\int_{x_0h(\theta x_0^{\alpha})}^{+\infty}(1+\theta t^{\alpha})^{-\frac12}t^{\alpha-3-k}T_{k+1}(\theta t^{\alpha})e^{-2it}dt+O(\theta^{1+\frac{3+k}{\alpha}}).$$
Concentrate now on $J_0:=\int_{x_0h(\theta x_0^{\alpha})}^{+\infty}(1+\theta t^{\alpha})^{-\frac12}t^{\alpha-3-k}T_{k+1}(\theta t^{\alpha})e^{-2it}dt$.
We check that, for $\alpha-2-k<0$,  $x_0h(\theta x_0^{\alpha})-x_0=O(\theta)$, and $t^{\alpha-3-k}\in L^1([x_0, +\infty[)$. Using the dominated convergence theorem, the limit of $J_0$ is easily obtained as being $\int_{x_0}^{+\infty}t^{\alpha-3-k}T_{k+1}(0)e^{-2it}dt$. To obtain the difference between $J_0$ and its limit, we perform integrations by parts on the integral defining $J_0$, in order to have $\alpha-2-m$ where $\alpha-2-m<-\alpha-2$, such that this difference is of order $\theta$. As it uses (\ref{eq:formula-ipp}) with a ``simplified" $\mathcal{L}$, we do not write the details. Finally,

\begin{equation}
\theta^{-1}J-\frac{1}{(2i)^{k+1}}\int_{x_0}^{+\infty}t^{\alpha-3-k}T_{k+1}(0)e^{-2it}dt=O(\theta).
\end{equation}
This proves the inequality
$$\int_{x_0}^{+\infty}\theta^{-1}M(y)e^{-2i\phi(y)}dy-\frac{1}{2i}T_k^S(x_0^{-1},0)-\frac{1}{(2i)^{k+1}}\int_{x_0}^{+\infty}t^{\alpha-3-k}T_{k+1}(0)e^{-2it}dt=O(\theta)$$
hence, for $\alpha>1$, using the inequality (\ref{in:alpha})
\begin{equation}\frac{R_{x_0}}{\theta}-\frac{1}{2i}T_k^S(x_0^{-1},0)-\frac{1}{(2i)^{k+1}}\int_{x_0}^{+\infty}t^{\alpha-3-k}T_{k+1}(0)e^{-2it}dt=O(\theta^{\frac{1}{\alpha}}).
\end{equation}
We thus proved Theorem \ref{th:refl:coeff}.

\section{The case $\alpha = 1$}
\label{sec-ramp}
%
%
%
%
%
As an illustration of the results of this paper, we recover the result of Theorem \ref{th:refl:coeff} in the model case
$$c^{-2}(x)=1+\theta x_+$$ 
associated with the equation
\begin{equation}\label{model-ramp}u''+(1+\theta x_+)u=0, u\in C^1.\end{equation} 
This case is easier, because we have an exact representation of the solutions in $[0, +\infty)$ using special functions, hence we may use the classical asymptotic results known for these functions.

Let us use $Ai$ and $Bi$ for the two classical solutions of the Airy equation $u''=xu$ (namely $Ai\in \mathcal{S}'(\R)$ is the inverse Fourier transform of $e^{i\frac{t^3}{3}}$ and $Bi$ is the unique solution of $u''(x)=xu$ such that $Bi(0)=\sqrt{3}Ai(0)$ and $Bi'(0)=-\sqrt{3}Ai'(0)$, see \cite{AS}).  

Introduce $w_{\pm}(X)=Ai( e^{\pm i\frac{\pi}{3}}X)$ a pair of satisfactory solutions\footnote{Equal to $u_{\pm}$ up to a multiplicative constant.} of $U''=-XU$, and recall that for $X\in \CC$, $\vert X\vert $ large:
$$Ai(X)\simeq \frac{1}{2} \pi^{-\frac12}X^{-\frac14}e^{-\frac23 X^{\frac32}}.$$
A pair of independent solutions of  (\ref{model-ramp})  in $x>0$ is 
 $(w_+(\beta (1+\theta x)), w_-(\beta (1+\theta x)))$, where $\beta^3\theta^2=1$, $\beta>0$ (that is $\beta=\theta^{-\frac23}$). In the region $x>0$, one has
 $$u(x)=Aw_+(\beta (1+\theta x))+Bw_-(\beta(1+\theta x)).$$

 We show in this paragraph that the decomposition of solutions of (\ref{model-ramp}) on $w_+$ and $w_-$ is the suitable decomposition to study outgoing and incoming solutions at $+\infty$. Indeed, we have 
 \begin{proposition} \label{plusinfini}The function $w_-$ is in the space of incoming solutions from $+\infty$,  and the function $w_+$ is in the space of outgoing solutions to $+\infty$ for (\ref{model-ramp}).
 
 The reflection coefficient $R$ is given by 
 $$R=\frac{w'_+(\beta)+i\beta^{\frac12}w_+(\beta)}{-w'_+(\beta)+i\beta^{\frac12}w_+(\beta)}.$$
 \end{proposition}
 \begin{proof}
  Consider $x\rightarrow w_{\pm}((1+i\sigma)^{\frac23}\beta (1+\theta x))$. These two functions are solution of (\ref{eq:model-sigma}), and in addition, converge, pointwise,  when $\sigma \rightarrow 0_-$ to $x\rightarrow  w_{\pm}(\beta (1+\theta x))$.
 
  The deformation $(1+\theta x)\rightarrow (1+i\sigma)(1+\theta x)$ transforms $\beta$ into ${\tilde \beta}$ such that ${\tilde \beta}^3\theta^2=(1+i\sigma)^2$, that is  ${\tilde \beta}= \beta (1+i\sigma)^{\frac23}$, and $-\frac23X^{\frac32}$  (which is the argument of the phase of $Ai$) is equal to $({\tilde \beta}(1+\theta x))^{\frac32}= (1+i\sigma)\beta^{\frac32}(1+\theta x)^{\frac32}$.Use 10.4.59 of \cite{AS}. Noticing that $\beta\theta>0$, and observing that
$$\Re (-\frac23 (e^{\pm i\frac{\pi}{3}}\theta^{\frac13}(x+\theta^{-1})(1+i\sigma)^{\frac23})^{\frac32}=\pm \frac23 \sigma (\theta^{\frac13}(x+\theta^{-1}))^{\frac32},$$
one checks that $\vert Ai(e^{-i\frac{\pi}{3}}\theta^{\frac13}
(1+i\sigma)^{\frac23}(x+\theta^{-1}))\vert\rightarrow +\infty$ for $\sigma<0$, hence for each $\sigma<0$, the family of solutions of (\ref{eq:model-sigma}) which go to zero when $x\rightarrow +\infty$ is 
$$A_+Ai(e^{i\frac{\pi}{3}}\theta^{\frac13}(1+i\sigma)^{\frac23}(x+\theta^{-1})).$$
This describes the space of solutions of (\ref{eq:model-problem}) which are outgoing at $+\infty$, that is $x\rightarrow w_+(\beta(1+\theta x))$ generates the space $\mathcal{U}^{>}_{+\infty}$.
 
 
 Since
 \[
 \frac{w_+(\beta(1 + \theta x))}{(1+\theta x)^{-\frac{1}{4}} e^{- \frac{2}{3} \theta^{-1/2} (1+\theta x)^{3/2}}} := \frac{w_+(\beta(1 + \theta x))}{b(x) e^{-i \varphi(x)}}
 \]
 tends to $\frac{1}{2} \pi^{-\frac{1}{2}} e^{- i \frac{\pi}{12}} \theta^{\frac{1}{6}}$ as $x \to \infty$, the unique normalized outgoing solution (in $\mathcal{U}^{>}_{+\infty}$) is 
 \[
 2 \pi^{\frac{1}{2}} e^{i \frac{\pi}{12}} \theta^{- \frac{1}{6}} w_+(\beta(1+\theta x)).
 \]
 
 The outgoing solution to $-\infty$ of (\ref{model-ramp}) is $e^{-ix}$, the incoming solution from $-\infty$ of (\ref{model-ramp}) is $e^{ix}$, hence one has:\\
 $\bullet$ in $]-\infty, 0[$, $u(x)= e^{-ix}+Re^{ix}$,\\
 $\bullet$ in $]0, +\infty[$, $u(x)= Tw_+(\beta (1+\theta x))$.\\
 The solution is $C^1(\mathbb R)$, hence 
 $$\left\{\begin{array}{l}1+R=Tw_+(\beta)\cr
 i(R-1)=\beta \theta Tw'_+(\beta)\end{array}\right.$$
 The relation $\beta\theta= \beta^{-\frac12}$ ends the proof of Proposition \ref{plusinfini}. \end{proof}
 Using estimates on the Airy functions, one proves now
 \begin{lemma} There exists $\theta_0$ and $C$ such that, for $0<\theta<\theta_0$
 \label{estimation-layer}
  \begin{equation} \vert R-\frac{i\theta}{8}\vert\leq C\theta^2.\end{equation}
  \end{lemma}
  \begin{proof}
The asymptotic expansion 
 $$\ln Ai(u)= \ln C-\frac14 \ln u-\frac23 u^{\frac32}+\ln (\sum_{k=0}^{\infty} (-1)^kc_k(\frac23 u^{\frac{3}{2}})^{-k}), \vert u \vert >1$$
leads to
 $$\frac{Ai'(u)}{Ai(u)}= -\sqrt{u}-\frac{1}{4u}+u^{\frac12}\frac{\sum_{k=1}^{\infty} (-1)^{k+1}kc_k(\frac23 u^{\frac{3}{2}})^{-k-1}}{\sum_{k=0}^{\infty} (-1)^kc_k(\frac23 u^{\frac{3}{2}})^{-k}}.$$
From this expression, and for $\vert u\vert \geq 1$, we get the inequality
 \be\label{ineq:airy}\vert\frac{Ai'(u)}{Ai(u)}+\sqrt{u}+\frac{1}{4u}\vert\leq \frac{C}{u^{\frac52}}\ee
 which is valid also in the complex region $\vert \arg(u)\vert <\pi$. An estimate of $R$ can then be obtained as
 $$R=\frac{\frac{w'_+(\beta)}{w_+(\beta)}+i\beta^{\frac12}}{ -\frac{w'_+(\beta)}{w_+(\beta)}+i\beta^{\frac12}}= \frac{e^{-i\frac{\pi}{3}}\frac{w'_+(\beta)}{w_+(\beta)}+\sqrt{e^{i\frac{\pi}{3}}\beta}}{ -e^{-i\frac{\pi}{3}}\frac{w'_+(\beta)}{w_+(\beta)}+\sqrt{e^{i\frac{\pi}{3}}\beta}}.$$
 \end{proof}
 
 The resulting estimate in $R$ turns out to be a particular case of \ref{th:refl:coeff} and the value of $R$ for $\alpha\rightarrow 1_+$ matches the one obtained directly for $\alpha=1$. For $1\leq \alpha<2$, the reflection operator is of order $-\alpha$, where $\alpha>1$ is the fractional regularity of $c_0^{-2}(x_1)$ at $x_1=0$. In the case $\alpha=1$ it  can also be expressed through Jost integrals.

\appendix
\section{Proof of the existence of polarized waves}
\label{sec:appendixA}
The existence of polarized waves (and the fact that the dimension of
${\mathcal U}^{>}_{+\infty}$ and of ${\mathcal U}^{>}_{+\infty}$ is 1)
rely on Lemma \ref{coddington} below. As $c>0$, define the new
variable $y=\phi(x)$ such that $y(0)=0$ and $\frac{dy}{dx}= c^{-1} (x)$. 
The equation (\ref{eq:model-sigma}) is equivalent to the equation on $U(y)= c^{-\frac12}(x(y))u(x(y))$
\be \partial^2_yU=[(\sigma -i)^2+\epsilon(y)]U\label{eq:reduit-sigma},\ee
where $\epsilon(y)=c(x(y))M(x(y))$.  A necessary condition which ensures that the behavior of the solutions of (\ref{eq:reduit-sigma}) is well predicted by the limiting system at infinity is 
\\
\centerline{ (H) There exists $y_0\geq 0$ such that $\int_{y_0}^{+\infty}\vert \epsilon(y)\vert dy<+\infty$. }\\
We note that
$$\int_{y_0}^{+\infty}\vert \epsilon(y)\vert dy= \int_{y_0}^{+\infty}c(x(y))\vert M(x(y))\vert dy= \int_{x(y_0)}^{+\infty}c(x)\vert M(x)\vert c^{-1}(x)dx=\int_{x(y_0)}^{+\infty}\vert M(x)\vert dx,$$
\begin{lemma}
\label{coddington}
Let $\sigma<0$. Under the assumption (H), the system $$\partial_y\left(\ba{c}U\cr
\partial_yU\ea\right)=\left(\ba{cc}0&1\cr
(\sigma-i)^2+\epsilon(y)&0\ea\right)\left(\ba{c}U\cr
\partial_yU\ea\right)$$  satisfies the conditions of Theorem 8.1 of chapter 3 of Coddington-Levinson, hence there exists a solution $W^{\sigma}_+(y)$  of it such that
$$W^{\sigma}_+(y)e^{(i-\sigma)y}\rightarrow \left(\ba{c}1\cr
\sigma-i\ea\right) , y\rightarrow +\infty,$$ and a solution $W^{\sigma}_-(y)$ of it such that $$W^{\sigma}_-(y)e^{(\sigma-i)y}\rightarrow \left(\ba{c}1\cr
-\sigma+i\ea\right) , y\rightarrow +\infty.$$
The solutions of this system are $AW_+^{\sigma}(y)+BW_-^{\sigma}(y)$.
\end{lemma}
We may thus deduce
\begin{cor}
\label{cod}
Let $\sigma<0$. The space of solutions of (\ref{eq:model-sigma}) which
tend to 0 at $+\infty$ is of dimension 1, generated by
$c^{\frac12}(x)U_+^{\sigma}(y(x))$, where $U_+^{\sigma}$ is the first
component of $W^{\sigma}_+$ . It is the same as the space of
solutions of (\ref{eq:model-sigma}) which are bounded on $[x_0, +\infty[$.
 
\end{cor}

Note that
$\partial^2_yU=[(\sigma -i)^2+\epsilon(y)]U$ is equivalent to the
system of Lemma \ref{coddington}. All solutions of this system are
$W(y)=AW_+^{\sigma}(y)+BW_-^{\sigma}(y)$, because $W_+^{\sigma}$ and
$W_-^{\sigma}$ are linearly independant. Hence $u(x)=c^{\frac12}(x)[AU_+^{\sigma}(y(x))e^{(i-\sigma)\phi(x)}e^{2(\sigma-i)\phi(x)}+BU_-^{\sigma}(y(x))e^{(\sigma-i)\phi(x)}]e^{(i-\sigma)\phi(x)}$. Hence the limit of $\vert u(x)\vert$ is not finite if $B\not=0$. Hence if $u$ has a finite limit then $B=0$, hence $u(x)=Ac^{\frac12}(x)U_+^{\sigma}(y(x))$. This function goes to 0 when $x$ goes to $+\infty$. The corollary is proved.\\
\section{Properties of the function $M$}
\label{sec:appendixB}
A key lemma for the normal convergence of the Volterra series of Section \ref{sec:Volterra} is the following
\begin{lemma} \label{ineq:Mx0}For all $x_0>0$, there exists $\theta_0(x_0)$ such that, for all $\theta<\theta_0(x_0)$, $M_{x_0}<2$.
\end{lemma}
It is a consequence of
\begin{lemma} For all $x_0$, there exists $\theta_*$ and $I_*$ such that for all $\theta<\theta_*$
\label{bound-Mx0} $$\ba{l}M_{x_0}\leq \theta^{\frac{1}{\alpha}}I_*, \alpha>1\cr
M_{x_0}\leq \theta I_*, 0<\alpha<1.\ea$$
\end{lemma}
\paragraph{Proof}One obtains that $M$, defined by (\ref{def:M}), is equal to
$$M(x)=-\frac14 \theta \alpha x_+^{\alpha-2}c^3[\alpha-1-\frac54 \theta \alpha c^2x_+^{\alpha}]:=-\theta x_+^{\alpha-2}T_0(\theta x_+^{\alpha})$$
where $T_0(X)=(\frac{\alpha(\alpha-1)}{4}-\alpha^2\frac{X}{1+X})(1+X)^{-\frac32}$
in the case $\lambda(x)=x$. This expression holds also for the inhomogeneous case (\ref{eq:c_lambdaofx}) (changing the function $T_0$ accordingly:
for $\lambda(x)=\int_0^x\chi(y)dy$ hence $b(x)=B_{1-\eta^2}(\theta x^{\alpha})$, 
\begin{equation}
T_0(X)=\frac14 (B_{1-\eta^2}(X))^6[\alpha(\alpha-1)\chi(X(1-\eta^2))+\alpha X(1-\eta^2)[\chi'(X(1-\eta^2))-\frac54 (\chi(X(1-\eta^2)))^2(B_{1-\eta^2}(X))^4]],
\end{equation}
which is compactly supported).\\
One obtains, for $\alpha>1$
$$\int_{x_0}^{+\infty}\vert M(y)\vert dy=\theta^{\frac{1}{\alpha}}I.$$
The integral $I$ is equal to $\int_{x_0 \theta^{\frac{1}{\alpha}}}^{+\infty} z^{\alpha-2}\vert T_0(z^{\alpha})\vert dz$.
As $\alpha>1$ and as there exists a constant $C$ such that $z^{\alpha-2}T_0(z^{\alpha})\leq Cz^{-\frac12\alpha-2}$ for $z>1$, the integral  $\int_0^{+\infty} z^{\alpha-2}\vert T_0(z^{\alpha})\vert dz$ is finite and is a majorant of $I$ (in the case $\lambda(x)=x$). In the case (\ref{eq:c_lambdaofx})  a majorant of $I$ is $\int_0^{z_0}z^{\alpha-2}\vert T_0(z^{\alpha})\vert dz$. There exists a constant $I_*$ such that
$$\int_{x_0}^{+\infty}\vert M(y)\vert dy\leq \theta^{\frac{1}{\alpha}}I_*.$$
In the case $0<\alpha<1$, one has $\vert M(x)\vert\leq \theta x_+^{\alpha-2}\mbox{max}\vert T_0\vert$ hence for $\alpha<1$
$$\int_{x_0}^{+\infty}\vert M(y)\vert dy\leq \theta \frac{x_0^{\alpha-1}}{1-\alpha}\mbox{max}\vert T_0\vert.$$
Together, these two inequalities imply Lemma \ref{ineq:Mx0}.
\section{Estimate of the sequence $s^{>}_{j}$}
Recall that one constructs a sequence $s^{>}_{j}$ such that
$$s^{>}_{j+1}(x)=\frac{1}{2i}[\int_{x_0}^x M(y)s^{>}_{j}(y)dy+\int_x^{+\infty}M(y)s^{>}_{j}(y)e^{-2i\phi(y)+2i\phi(x)}dy]$$
hence satisfying
$$c\frac{d s^{>}_{j+1}}{dx}=\int_x^{+\infty}M(y)s^{>}_{j}(y)e^{-2i\phi(y)+2i\phi(x)}dy.$$
Recall that $\mathcal{L}$ is given by (\ref{mathcalL}).

The first tool is the following \begin{lemma}
\label{l:A1}
Assume that the function $m$ is regular enough at $+\infty$, meaning that $\mathcal{L}^k(m)$ is integrable and goes to 0 as $y\rightarrow +\infty$ for all $k\geq 0$.\\
One has the identity, for all $k$ and $j$
$$\begin{array}{ll}\int_{x}^{+\infty}m(y)s^{>}_{j}(y)e^{-2i\phi(y)}dy&=\frac{cs^{>}_{j}(x)}{2i}e^{-2i\phi(x)}[\sum_{p=0}^k\frac{\mathcal{L}^p(m)}{(2i)^p}]+\frac{1}{2i}\int_{x}^{+\infty}(\sum_{p=0}^k\frac{\mathcal{L}^p(m)}{(2i)^p})c\frac{ds^{>}_{j}}{dy}(y)e^{-2i\phi(y)}dy\cr
&+\frac{1}{(2i)^{k+1}}\int_x^{+\infty}\mathcal{L}^{k+1}(m)s^{>}_{j}(y)e^{-2i\phi(y)}dy.\end{array}$$
\end{lemma}
\begin{proof}
As $\phi'c=1$, one obtains
$$\int_x^{+\infty} m(y)s^{>}_{j}(y)e^{-2i\phi(y)}dy= -\frac{1}{2i}\int_x^{+\infty}cm(y)s^{>}_{j}(y) \frac{d}{dy}(e^{-2i\phi(y)})dy.$$
If $m$ and $\mathcal{L}(cm)$ belong to $L^1([x_0, +\infty[)$, and $m$ goes to 0 when $y\rightarrow +\infty$, one obtains:
$$\int_x^{+\infty} m(y)s^{>}_{j}(y)e^{-2i\phi(y)}dy=e^{-2i\phi(x)}\frac{cm(x)s^{>}_{j}(x)}{2i}+\frac{1}{2i}\int_x^{+\infty}\mathcal{L}(ms^{>}_{j})(y) e^{-2i\phi(y)}dy,$$
and using $\mathcal{L}(ms^{>}_{j})(y)=\mathcal{L}(m)s^{>}_{j}+m(y)c\frac{ds^{>}_{j}}{dy}(y)$, one obtains Lemma (\ref{l:A1}) for $k=0$.

Using 

$$\mathcal{L}(m)(y) s^{>}_{j}e^{-2i\phi(y)}dy=-\frac{1}{2i}\mathcal{L}(m)(y) cs^{>}_{j} (y)\frac{d}{dy}(e^{-2i\phi(y)})$$
one obtains
$$\int_x^{+\infty}\mathcal{L}(m)(y) s^{>}_{j}e^{-2i\phi(y)}dy=\frac{1}{2i}\mathcal{L}(m)(x) cs^{>}_{j}(x)e^{-2i\phi(x)}+\frac{1}{2i}\int_x^{+\infty}\mathcal{L}[\mathcal{L}(m) s^{>}_{j}] (y)e^{-2i\phi(y)}dy$$
and $\mathcal{L}[\mathcal{L}(m) s^{>}_{j}](y)= \mathcal{L}^2(m)s^{>}_{j}(y)+\mathcal{L}(m)(y)c\frac{ds^{>}_{j}}{dy}(y)$ yields Lemma (\ref{l:A1}) for $k=1$.

One proceeds successively.
\end{proof}
An easy consequence is the identity
\begin{equation}
c\frac{ds_{>, 1}}{dy}=\frac{c}{2i}[\sum_{p=0}^k\frac{\mathcal{L}^p(M)}{(2i)^p}]+\frac{1}{(2i)^{k+1}}e^{2i\phi(x)}\int_x^{+\infty}\mathcal{L}^{k+1}(M)e^{-2i\phi(y)}dy\label{eq:ds1}
\end{equation}
as well as the generalization, $j\geq 2$
\begin{equation}
\begin{array}{ll}c\frac{ds_{>, j}}{dy}&=\frac{cs_{>,j-1}}{2i}[\sum_{p=0}^k\frac{\mathcal{L}^p(M)}{(2i)^p}]+\frac{1}{2i}e^{2i\phi(x)}\int_{x}^{+\infty}(\sum_{p=0}^k\frac{\mathcal{L}^p(M)}{(2i)^p})c\frac{ds_{>,j-1}}{dy}(y)e^{-2i\phi(y)}dy\cr
&+\frac{1}{(2i)^{k+1}}e^{2i\phi(x)}\int_x^{+\infty}\mathcal{L}^{k+1}(M)e^{-2i\phi(y)}s_{>, j-1}(y)dy\end{array}
\label{eq:rec:si}
\end{equation}

\begin{lemma}
\label{l:A2}
i) There exist functions $T_p(X)$, uniformly bounded in $\theta$ for $X\in [0, +\infty[$ by $T_p^{\infty}$, such that 
\begin{equation}
\mathcal{L}^p(M)(x)=-\theta x^{\alpha-2-p}T_p(\theta x^\alpha)\label{symbolep}
\end{equation}
 In addition $T_{p+1}(0)=(\alpha-2-p)T_p(0), T_0(0)=\frac{\alpha (\alpha-1)}{4}$.\\
ii) If $\alpha-2-k<0$, one has the inequality
\begin{equation}
\vert\int_x^{+\infty}\mathcal{L}^{k+1}(M)e^{-2i\phi(y)}dy\vert\leq \theta \frac{x^{\alpha-2-k}}{k+2-\alpha}T_{k+1}^{\infty}
\end{equation}
iii) Consider the functions $T_k^S(x, X)=\sum_{p=0}^k x^{-p}\frac{T_p(X)}{(2i)^p}$. One has
$$\sum_{p=0}^k\frac{\mathcal{L}^p(M)}{(2i)^p}= -\theta x^{\alpha-2}T_k^S(x,\theta x^{\alpha})$$
with
\begin{equation}\label{estimee-somme-Lk}\vert T_k^S(x, \theta x^{\alpha})\vert \leq \sum_{p=0}^k \frac{T_p^{\infty}}{(2x_0)^p}.\end{equation}
iv) For $\alpha>2$, there exists $\theta_1(x_0)$ and a constant $C$ such that, for all $\theta<\theta_1(x_0)$ and $x\geq x_0$
\begin{equation}\label{eq:estimate-sum}
\vert \sum_{p=0}^k\frac{\mathcal{L}^p(M)}{(2i)^p}(x)\vert \leq C\theta^{\frac{2}{\alpha}}
\end{equation}
and there exists $x\in [x_0, +\infty)$ such that 
$$\vert \sum_{p=0}^k\frac{\mathcal{L}^p(M)}{(2i)^p}(x)\vert \theta^{-\frac{2}{\alpha}}\geq \frac12 C.$$
For $1<\alpha<2$, there exists $\theta_2$ such that, for $\theta<\theta_2$ and for all $x\geq x_0$
$$\vert \sum_{p=0}^k\frac{\mathcal{L}^p(M)}{(2i)^p}(x)\vert \leq \vert \sum_{p=0}^k\frac{\mathcal{L}^p(M)}{(2i)^p}(x_0)\vert\simeq C\theta$$
\end{lemma}
Note that item iv) of this Lemma proves that $\sum_{p=0}^k\frac{\mathcal{L}^p(M)}{(2i)^p}(x)$ is optimally of order $\theta^{\frac{2}{\alpha}}$ for $\alpha>2$ and of order $\theta$ for $1<\alpha<2$ on $[x_0, +\infty)$.
\paragraph{Proof of Lemma \ref{l:A2}}
The item i)  is a consequence of the induction relation $T_{p+1}(X)=(\alpha-2-p)T_p(X)+\alpha X T'_p(X)$ and $T_0(X)=(\frac{\alpha(\alpha-1)}{4}-\frac54\alpha^2\frac{X}{1+X})(1+X)^{-\frac32}$. One gets then
$$\mathcal{L}^{p+1}(M)=-\theta x^{\alpha-2-p}T_{p+1}(\theta x^{\alpha})$$ 
with $T_{p+1}(0)=(\alpha-2-p)T_p(0)$. In addition, using $T_p(X)= (1+X)^{-p-\frac52}Q_p(X)$ where $Q_p$ is a polynomial of degree $p+1$ (which is also proven by the induction relation)

The proof of the second and of the third item comes from $\vert T_p(X)\vert\leq C_p(1+X)^{-\frac32}$, decreasing at $+\infty$ as well as all its derivatives, smooth for $X\in [0, +\infty[$ in the case $\lambda(x)=x$ and bounded uniformly as well as all its derivatives for $\lambda(x)=\int_0^x\chi(y)dy$.

The proof of item iv) comes from, for $\alpha>2$,
$$\sum_{p=0}^k\frac{\mathcal{L}^p(M)}{(2i)^p}(x)=\theta^{\frac{2}{\alpha}}\sum_{k=0}^p \theta^{\frac{p}{\alpha}}z^{\alpha-2-p}T_p(z^{\alpha}),\theta x^{\alpha}=z^{\alpha}.$$
One observes that $\vert z^{\alpha-2}T_0(z^{\alpha})\vert $ is maximum when $(\alpha-2) T_0(z^{\alpha})+\alpha z^{\alpha}T'_0(z^{\alpha})=0$. This equation have roots, the one leading to the point of maximum of $\vert z^{\alpha-2}T_0(z^{\alpha})\vert $ being denoted by $z_*$. The value of this maximum is $C_*=\vert z_*^{\alpha-2}T_0(z_*^{\alpha})\vert >0$. Using the implicit function theorem when $\theta\rightarrow 0$, there exists a point of maximum of $\sum_{k=0}^p \theta^{\frac{p}{\alpha}}z^{\alpha-2-p}T_p(z^{\alpha})$ close to $z_*$, and the maximum is close to $\vert z_*^{\alpha-2}T_0(z_*^{\alpha})\vert$. Hence, as the limit of $\vert z^{\alpha-2}T_0(z^{\alpha})\vert$ is zero when $z\rightarrow +\infty$, there exists $Z_1>z_*$ such that $\vert Z_1^{\alpha-2}T_0(Z_1^{\alpha})\vert =\frac{C_*}{2}$.

For $1<\alpha<2$, the limit of $z^{\alpha-2}T_0(z^{\alpha})$ when $z\rightarrow 0$ is $+\infty$, while its limit is 0 when $z\rightarrow +\infty$. Again, there is a unique point of minimum of $z^{\alpha-2}T_0(z^{\alpha})$, which is a point of maximum of $\vert z^{\alpha-2}T_0(z^{\alpha})\vert$. For $\theta$ small enough, the maximum value of $\vert \sum_{p=0}^k\frac{\mathcal{L}^p(M)}{(2i)^p}(x)\vert $ is obtained for $x=x_0$. In addition, as for $\alpha=2$, $T_0(1)=0$ and the limit of $T_0$ at $+\infty$ being zero, there exists a point of maximum of $\vert T_0(z^2)\vert$ in $[1, +\infty[$.

We may prove using Lemmas \ref{l:A1} and \ref{l:A2} that $s_{>,1}$ is no better than $\theta^{\frac{1}{\alpha}}$ for $\alpha>1$
\begin{lemma}
\label{l:A1bis}
There exists a constant $C_1$ such that, for $\alpha>2$
$$\vert 2is_{>,1}(x)-\int_{x_0}^x M(y)dy\vert \leq C_1\theta^{\frac{2}{\alpha}}, x\geq x_0$$
and for $1<\alpha\leq 2$
$$\vert 2is_{>,1}(x)-\int_{x_0}^x M(y)dy\vert \leq C_1\theta, x\geq x_0$$
\end{lemma}
Note that Lemma \ref{l:A1bis} proves that also that the estimate $\vert\vert s_{>,1}\vert\vert_{\infty}=O(\theta^{\frac{1}{\alpha}}$ is optimal.
\paragraph{Proof of Lemma \ref{l:A1bis}}
Let $k=[\alpha-1]$. Consider $j=0$ in the identity of Lemma \ref{l:A1} and $m(y)=M(y)$. We check that
$$-\big[2is_{>,1}(x)-\int_{x_0}^xM(y)dy+\frac{c}{2i}[\sum_{p=0}^{k}\theta x^{\alpha-2-p}T_p(\theta x^\alpha)]\big]=\int_{x}^{+\infty}\theta y^{\alpha-3-k}T_{k+1}(y)e^{2i(\phi(x)-\phi(y)}dy.$$
As $\alpha-3-[\alpha-1]<-1$, one has, for $x\geq x_0$,
$$\vert \int_{x}^{+\infty}\theta y^{\alpha-3-k}T_{k+1}(y)e^{2i(\phi(x)-\phi(y)}dy\vert\leq \theta T_{[\alpha]}^{\infty}\frac{x^{\alpha-[\alpha]-1}}{[\alpha]+1-\alpha}\leq \theta T_{[\alpha]}^{\infty}\frac{x_0^{\alpha-[\alpha]-1}}{[\alpha]+1-\alpha}.$$
Hence there exists $C_0>0$ such that
$$\vert 2is_{>,1}(x)-\int_{x_0}^xM(y)dy+\frac{c}{2i}[\sum_{p=0}^{k}\theta x^{\alpha-2-p}T_p(\theta x^\alpha)]\vert\leq C_0\theta.$$
For $\alpha> 2$, one obtains, thanks to $\frac{2}{\alpha}<1$, that there exists $\theta_2$ such that for $\theta\leq \theta_2$, using estimate (\ref{eq:estimate-sum})
$$\vert 2is_{>,1}(x)-\int_{x_0}^xM(y)dy\vert\leq C_2\theta^{\frac{2}{\alpha}}.$$
For $1<\alpha\leq 2$, one obtains, for $\theta\leq \theta_2$
$$\vert 2is_{>,1}(x)-\int_{x_0}^xM(y)dy\vert\leq C_2\theta.$$
Remark finally that
$\int_{x_0}^x M(y)dy=\theta^{\frac{1}{\alpha}}\int_{x_0\theta^{\frac{1}{\alpha}}}^{x\theta^{\frac{1}{\alpha}}}z^{\alpha-2}T_0(z^{\alpha})dz$, hence $\int_{x_0}^xM(y)dy$ is of order $\theta^{\frac{1}{\alpha}}$.
\begin{lemma}
For each given $m_0$,  for all $j$,  there exists $j$ functions $A_j^l$ (depending on $m_0$) and $B^l_j$ such that $A^l_j(x,\theta)=(\theta x^{\alpha-2})^{j-l}B^l_j(x^{-1}, \theta x^{\alpha})$ such that the function $r_j$ given by
\begin{equation}\label{eq:dsj}r_j(x)=c\frac{ds_{>, j}}{dx}-\sum_{l=0}^{j-1}A^l_j(x, \theta)s_{>,l}(x)
\end{equation}
satisfies, for all $x\geq x_0$, $$\vert r_j(x)\vert\leq \theta C_{m_0}^jjx^{\alpha-m_0}.$$
for a constant $C_{m_0}^j$
\label{l:A3}
\end{lemma}
The proof is done by induction, and we begin by the two first steps explicitely to show the way of obtaining $A_j^l$ and $r_j$.

We observe first $c\frac{ds_{>,1}}{dy}=\theta A_1^0(x)+r_1(x)$, with $A_1^0(x)=\frac{1}{2i}x^{\alpha-2}T_k^S(x,\theta x^{\alpha})c(\theta x^{\alpha})$ and
$$r_1(x)=\frac{1}{(2i)^{k+1}}e^{2i\phi(x)}\int_x^{+\infty} \mathcal{L}^{k+1}(M)e^{-2i\phi(y)}dy.$$

From the identity of Lemma \ref{l:A1} written for $k=k_2$ and the equality (\ref{eq:ds1}) written for $k=k_1$ one obtains
\begin{equation}
\label{eq:ds2}\begin{array}{ll}c\frac{ds_{>,2}}{dy}=&\frac{cs_{>,1}}{2i}[\sum_{p=0}^{k_2}\frac{\mathcal{L}^p(M)}{(2i)^p}]+\frac{1}{2i}e^{2i\phi(x)}\int_x^{+\infty}(\sum_{p=0}^{k_2}\frac{\mathcal{L}^p(M)}{(2i)^p})\frac{c}{2i}[\sum_{p=0}^{k_1}\frac{\mathcal{L}^p(M)}{(2i)^p}]e^{-2i\phi(y)}dy\cr
&+\frac{1}{(2i)^{k_1+k_2+2}}e^{2i\phi(x)}\int_x^{+\infty}(\sum_{p=0}^{k_2}\frac{\mathcal{L}^p(M)}{(2i)^p})\int_y^{+\infty}\mathcal{L}^{k_1+1}(M)(z)e^{-2i\phi(z)}dzdy\cr
&+\frac{1}{(2i)^{k_2+1}}e^{2i\phi(x)}\int_x^{+\infty}\mathcal{L}^{k_2+1}(M)e^{-2i\phi(y)}s_{>, 1}(y)dy\end{array}\end{equation}
This rewrites
$$c\frac{ds_{>,2}}{dy}=\theta A_2^1(y)s_{>,1}(y)+\theta^2A_2^0(y)+r_2(y)$$
where $A_2^1=A_1^0$ (for the same choice of $k$, number of terms of the expansion) and $A_2^0(y)=c(\theta y^{\alpha})\sum_{l=0}^{k}\frac{\mathcal{L}^l(m)}{(2i)^l}$, $m(y)=\frac{c}{2i}y^{2\alpha-4}T_{k_1}^S(x, \theta x^{\alpha})T_{k_2}^S(x^{-1}, \theta x^{\alpha})$. One then obtains
$$A_2^0(y)=(\theta y^{\alpha-2})^2B_{2,0}(y^{-1}, \theta y^{\alpha})$$
Using (\ref{eq:rec:si}) for $i=3$, one has

$$
\begin{array}{ll}c\frac{ds_{>, 3}}{dy}&=\frac{cs_{>,2}}{2i}[\sum_{p=0}^{k_3}\frac{\mathcal{L}^p(M)}{(2i)^p}]+\frac{1}{2i}e^{2i\phi(x)}\int_{x}^{+\infty}(\sum_{p=0}^{k_3}\frac{\mathcal{L}^p(M)}{(2i)^p})c\frac{ds_{>,2}}{dy}(y)e^{-2i\phi(y)}dy\cr
&+\frac{1}{(2i)^{k_3+1}}e^{2i\phi(x)}\int_x^{+\infty}\mathcal{L}^{k_3+1}(M)e^{-2i\phi(y)}s_{>, 2}(y)dy\end{array}
$$

Using (\ref{eq:ds2}) in this equality, one obtains

$$
\begin{array}{ll}c\frac{ds_{>, 3}}{dy}&=\frac{cs_{>,2}}{2i}[\sum_{p=0}^{k_3}\frac{\mathcal{L}^p(M)}{(2i)^p}]+\frac{1}{2i}e^{2i\phi(x)}\int_{x}^{+\infty}(\sum_{p=0}^{k_3}\frac{\mathcal{L}^p(M)}{(2i)^p})[\sum_{p=0}^{k_2}\frac{\mathcal{L}^p(M)}{(2i)^p}](y)\frac{cs_{>,1}}{2i}e^{-2i\phi(y)}dy\cr
&+\frac{1}{(2i)^2}e^{2i\phi(x)}\int_{x}^{+\infty}(\sum_{p=0}^{k_3}\frac{\mathcal{L}^p(M)}{(2i)^p})\int_y^{+\infty}(\sum_{p=0}^{k_2}\frac{\mathcal{L}^p(M)}{(2i)^p})\frac{c}{2i}[\sum_{p=0}^{k_1}\frac{\mathcal{L}^p(M)}{(2i)^p}]e^{-2i\phi(z)}dzdy\cr
&+\frac{1}{(2i)^{k_2+2}}e^{2i\phi(x)}(\sum_{p=0}^{k_3}\frac{\mathcal{L}^p(M)}{(2i)^p})\int_y^{+\infty}\int_x^{+\infty}\mathcal{L}^{k_2+1}(M)e^{-2i\phi(z)}s_{>, 1}(z)dzdy\cr
&+\frac{1}{(2i)^{k_1+k_2+4}}e^{2i\phi(x)}\int_x^{+\infty}(\sum_{p=0}^{k_3}\frac{\mathcal{L}^p(M)}{(2i)^p})\int_y^{+\infty}(\sum_{p=0}^{k_2}\frac{\mathcal{L}^p(M)}{(2i)^p})(y)\int_z^{+\infty}\mathcal{L}^{k_1+1}(M)(t)e^{-2i\phi(t)}dtdzdy\cr&+\frac{1}{(2i)^{k_3+1}}e^{2i\phi(x)}\int_x^{+\infty}\mathcal{L}^{k_3+1}(M)e^{-2i\phi(y)}s_{>, 2}(y)dy\end{array}
$$
We have to analyse all terms of this equality. The three last terms are uniformly bounded as soon as $k_1,k_3$ are large enough, owing to the regularizing effect of $\mathcal {L}$ characterized through Lemma \ref{l:A2}. Let us consider, for example, the term 
$$\frac{1}{(2i)^{k_1+k_2+4}}e^{2i\phi(x)}\int_x^{+\infty}(\sum_{p=0}^{k_3}\frac{\mathcal{L}^p(M)}{(2i)^p})\int_y^{+\infty}(\sum_{p=0}^{k_2}\frac{\mathcal{L}^p(M)}{(2i)^p})(y)\int_z^{+\infty}\mathcal{L}^{k_1+1}(M)(t)e^{-2i\phi(t)}dtdzdy$$
The estimate for this term uses (\ref{estimee-somme-Lk}) as follows:
$$\vert \int_z^{+\infty}\mathcal{L}^{k_1+1}(M)(t)e^{-2i\phi(t)}dt\vert\leq \theta T_{k_1+1}^{\infty}\frac{z^{\alpha-2-k_1}}{2+k_1-\alpha}, k_1>\alpha-2$$
from which one deduces
$$\begin{array}{ll}\vert \int_y^{+\infty}(\sum_{p=0}^{k_2}\frac{\mathcal{L}^p(M)}{(2i)^p})(y)\int_z^{+\infty}\mathcal{L}^{k_1+1}(M)(t)e^{-2i\phi(t)}dtdz\vert&\leq \theta T_{k_2+1}^{\infty}\theta T_{k_1+1}^{\infty}\frac{z^{2\alpha-4-k_1-1}}{3+k_1-\alpha}dz\cr
&=\theta^2T_{k_2+1}^{\infty}T_{k_1+1}^{\infty}\frac{y^{2\alpha-3-k_1}}{(3+k_1-2\alpha)(2+k_1-\alpha)}, k_1>2\alpha-3\end{array}$$
hence a final estimate of this remainder term by $C\theta^2x^{3\alpha-4-k_1}$ for $k_1>3\alpha-4$. When $\alpha>1$, the condition which contains all is $k_1>3\alpha-4$.\\
 The third term is an integral which do not contain any term of the form $s^{>}_{j}$, hence can be easily treated by integration by parts and leads to a term of the form $$(\theta y^{\alpha})^3 B_{3,0}(y^{-1}, \theta y^{\alpha}).$$
 The first term is the finite sum $A_2^1(y)$. The only term left to consider is the second term, on which we use Lemma \ref{l:A1}. The coefficient of $s_{>,1}$ in this integral is of the form $(\theta y^{\alpha-2})^2B_{2,0}(y^{-1}, \theta y^{\alpha})$, hence the action of $\sum_p\frac{\mathcal{L}^p}{(2i)^p}$ returns a similar term, which writes then
 $$(\theta y^{\alpha-2})^2B_{3,2}(y^{-1}, \theta y^{\alpha}).$$
 This allows us to write the existence of three functions $A^2_3, A^1_3$  and $A_3^0$, and a remainder term $r$, such that
$$c\frac{ds_{>, 3}}{dy}=A^2_3(x)s_{>,2}(x)+A^1_3(x)s_{>,1}(x)+A^0_3(x)+r_3(x)$$
where $\vert r_3(x)\vert \leq C^3_M\theta x^{\alpha-M}$ where $M$ is large enough.  \\
For the general proof by induction, let us assume that the Lemma is true for all $j'\leq j$. One uses Lemma \ref{l:A1} to obtain
$$\begin{array}{ll}c\frac{ds^{>}_{j+1}}{dy}&=\frac{cs^{>}_{j}}{2i}[\sum_0^k\frac{(\mathcal{L})^p(M)}{(2i)^p}]+\frac{1}{2i}e^{2i\phi(x)}\int_x^{+\infty}(\sum_0^k\frac{(\mathcal{L})^p(M)}{(2i)^p})c\frac{ds^{>}_{j}}{dy}e^{-2i\phi(y)}dy\cr
&+\frac{1}{(2i)^{k+1}}\int_x^{+\infty}\mathcal{L}^{k+1}(M)s^{>}_{j}(y)e^{-2i\phi(y)}dy.\end{array}$$
We plug the induction hypothesis in the equality to obtain
$$\begin{array}{ll}c\frac{ds^{>}_{j+1}}{dy}&=\frac{cs^{>}_{j}}{2i}[\sum_0^k\frac{(\mathcal{L})^p(M)}{(2i)^p}]+\frac{1}{2i}e^{2i\phi(x)}\int_x^{+\infty}(\sum_0^k\frac{(\mathcal{L})^p(M)}{(2i)^p})[\sum_{l=0}^{j-1}A_j^l(y)s_{>,l}(y)]e^{-2i\phi(y)}dy\cr
&+\frac{1}{2i}e^{2i\phi(x)}\int_x^{+\infty}(\sum_0^k\frac{(\mathcal{L})^p(M)}{(2i)^p})r_j(y)e^{-2i\phi(y)}dy+\frac{1}{(2i)^{k+1}}\int_x^{+\infty}\mathcal{L}^{k+1}(M)s^{>}_{j}(y)e^{-2i\phi(y)}dy.\end{array}$$
Hence the coefficient of $s^{>}_{j}$ comes from the first term only and is equal to $\frac{c}{2i}\sum_0^k\frac{(\mathcal{L})^p(M)}{(2i)^p}$. We denote it by $$A_{j+1}^j=\frac{1}{2i}x^{\alpha-2}c(\theta x^{\alpha})M_k^S(x, \theta x^{\alpha})$$

The identity of Lemma \ref{l:A1} yields, with the notation $m_j^l(x)=(\sum_0^k\frac{\mathcal{L}^p(M)}{(2i)^p})A_j^l(x)$ and $l\leq j-1$

$$\begin{array}{ll}\int_x^{+\infty}m_j^l(y)s_{>,l}(y)e^{-2i\phi(y)}dy&=\frac{cs_{>,l}(x)}{2i}e^{-2i\phi(x)}[\sum_{p=0}^k\frac{\mathcal{L}^p(m_j^l)}{(2i)^p}]+\frac{1}{2i}\int_{x}^{+\infty}(\sum_{p=0}^k\frac{\mathcal{L}^p(m_j^l)}{(2i)^p})c\frac{ds_{>,l}}{dy}(y)e^{-2i\phi(y)}dy\cr
&+\frac{1}{(2i)^{k+1}}\int_x^{+\infty}\mathcal{L}^{k+1}(m_j^l)s^{>}_{j}(y)e^{-2i\phi(y)}dy.\end{array}$$
We use the induction hypothesis for all terms of the equality above containing $c\frac{ds_{>,l}}{dy}$. We are thus left to evaluate, for $l\geq 1$
$$\frac{1}{2i}\int_{x}^{+\infty}(\sum_{p=0}^k\frac{\mathcal{L}^p(m_j^l)}{(2i)^p})c\frac{ds_{>,l}}{dy}(y)e^{-2i\phi(y)}dy=\frac{1}{2i}\int_{x}^{+\infty}(\sum_{p=0}^k\frac{\mathcal{L}^p(m_j^l)}{(2i)^p})[r_l(y)+\sum_{l'\leq l-1}A_l^{l'}(y)s_{>,l'}(y)]e^{-2i\phi(y)}dy.$$
On each individual term $\frac{1}{2i}\int_{x}^{+\infty}(\sum_{p=0}^k\frac{\mathcal{L}^p(m_j^l)}{(2i)^p})A_l^{l'}(y)s_{>,l'}(y)e^{-2i\phi(y)}dy$, one uses Lemma \ref{l:A1} which reduces the evaluation of this sum to the evaluation of $c\frac{ds_{>,l'}}{dy}$ for $l'\geq 1$ (because for $l'=0$, this term is equal to 0. This evaluation, thanks to identity (\ref{eq:rec:si}), involves only terms of the form $A_{l'}^{l''}s_{>,l''}$ with $l''\leq l'-1\leq l-2\leq j-3$. For each of these terms, one uses Lemma \ref{l:A1}. As the index $j,l,l',l''$ form a strictly decreasing sequence of integers, one is left with a finite iteration process. All terms involved are functions thanks to the application of Lemma \ref{l:A1} in the particular case $j=0$, because $s_{>,0}=1$.\\
The remainder terms in these successive equalities are either terms of the form
$$\int_x^{+\infty}\mathcal{L}^{k+1}(m_l^{l'})(y)s^{>}_{j}e^{-2i\phi(y)}dy$$
or
$$\int_x^{+\infty}r_l(y)m(y)e^{-2i\phi(y)}dy.$$
We use Lemma \ref{l:A2} to have explicit expressions for
$$\mathcal{L}^p(M), \sum_{p=0}^k\frac{\mathcal{L}^p(M)}{(2i)^p}, \sum_{p=0}^k\frac{\mathcal{L}^p(m_j^l)}{(2i)^p},$$
All these terms write, respectively $$-\theta x^{\alpha-2-p}T_p(\theta x^{\alpha}), -\theta x^{\alpha-2}T^S_k(\theta x^{\alpha}), -\theta x^{\beta-2}T_{j,l,k}(\theta, \theta x^\alpha)$$
where $\beta$ depends on $\alpha$ and $j,l,k$ and the expression of $T^S_k$ and of $T_{j,l,k}$ is complicated and not needed. From these explicit expressions, we obtain similar expressions for $\mathcal{L}^{k+1}(m_l^{l'})$, the factor being $\theta x^{\alpha-2-k-1}$. Using the explicit integration by parts with $\mathcal{L}$ and the bound of $T_p,T^S_k, T_{j,l,k}$ one deduces similar relations for the remainder terms.\\
The relation on $c\frac{ds_{>, j+1}}{dy}$ is thus obtained, which ends the proof of Lemma \ref{l:A3}. It is also a (difficult) consequence of this proof that
$$A_j^l(x)=(\theta x^{\alpha-2})^{j-l} B_{j,l}(x^{-1}, \theta x^{\alpha}).$$
Along with $\vert s^{>}_{j}(x)\vert\leq C_j\theta^{\frac{j}{\alpha}}$ for $x\in [x_0, +\infty[$, this helps to evaluate all the terms of the expansion of $c\frac{ds^{>}_{j}}{dx}$.\\
Note in particular that, for $x\in [x_0, +\infty[$ and $x_0>0$ fixed, all terms $B_{j,l}(x^{-1}, \theta x^{\alpha})$ are uniformly bounded for $\theta<\theta_0(x_0)$ given.\\
This estimate is crucial to allow us to use the dominated convergence theorem as follows, for $\alpha-2-p<-1$
$$\mbox{lim}_{\theta\rightarrow 0_+} \int_{x_0}^{+\infty}x^{\alpha-2-p}M_p(\theta x^{\alpha})e^{-2i\phi(x)}dx= \int_{x_0}^{+\infty}x^{\alpha-2-p}M_p(0)e^{-2ix}dx.$$
The condition $\alpha-2-p<-1$ is ensured by using the identity of Lemma \ref{l:A1} for $x=x_0$, $k$ large enough, $j=0$ for treating the term $\int_{x_0}^{+\infty}\mathcal{L}^{k+1}(m)e^{-2i\phi(y)}dy$. For example, one is left with
$$\begin{array}{ll}\mbox{lim}_{\theta\rightarrow 0_+}\int_{x_0}^{+\infty}\theta^{-1}M(y)e^{-2i\phi(y)}dy=& \frac{1}{2i}[\sum_{p=0}^k\frac{1}{(2i)^p}\frac{d^p}{dy^p}(-\frac{\alpha(\alpha-1)}{4}y^{\alpha-2})(x_0)]\cr
&+\frac{1}{(2i)^{k+1}}\int_{x_0}^{+\infty}\frac{d^{k+1}}{dy^{k+1}}(-\frac{\alpha(\alpha-1)}{4}y^{\alpha-2})e^{-2iy}dy.\end{array}$$
We thus deduce
\label{p:all-terms}
\begin{proposition}
For all $k$ large enough, there exists a function $r_j^k$ such that one has
$$\int_{x_0}^{+\infty}\theta^{-1}M(y)s^{>}_{j}(y)e^{-2i\phi(y)}dy=\frac{cs^{>}_{j}}{2i} x_0^{\alpha-2}T_k^S(x_0^{-1}, \theta x_0^\alpha)+\theta^{\frac{j}{\alpha}}r_j^k(x_0; \theta)+O(\theta).$$
One deduces, for $k$ large enough
$$\int_{x_0}^{+\infty}\theta^{-1}M(y)e^{-2i\phi(y)}(\sum_{j=0}^{n_0}s^{>}_{j})(y)dy=\frac{1}{2i} x_0^{\alpha-2}T_k^S(x_0^{-1}, \theta x_0^\alpha) +r_{n_0, k}(x_0,0)+O(\theta^{\frac{1}{\alpha}}).$$
\end{proposition}
Let us use Lemma \ref{l:A1}. The proof of this proposition reduces then to proving
$$\frac{1}{2i}\int_{x_0}^{+\infty}\theta y^{\alpha-2}M_k^S(y^{-1}, \theta y^{\alpha})c\frac{ds^{>}_{j}}{dy}e^{-2i\phi(y)}dy=O(\theta),$$
id est $\frac{1}{2i}\int_{x_0}^{+\infty} y^{\alpha-2}M_k^S(y^{-1}, \theta y^{\alpha})c\frac{ds^{>}_{j}}{dy}e^{-2i\phi(y)}dy$ bounded uniformly. Use Lemma \ref{l:A3}. One has
$$\begin{array}{c}\int_{x_0}^{+\infty} y^{\alpha-2}M_k^S(y^{-1}, \theta y^{\alpha})c\frac{ds^{>}_{j}}{dy}e^{-2i\phi(y)}dy\cr
=\int_{x_0}^{+\infty} y^{\alpha-2}M_k^S(y^{-1}, \theta y^{\alpha})[\sum_{l=0}^{j-1}(\theta y^{\alpha})^{j-l}B_{j,l}(y^{-1}, \theta y^{\alpha})s_{>,l}(y)]e^{-2i\phi(y)}dy+r\cr
=\sum_{l'=1}^j\theta^{l'}\int_{x_0}^{+\infty} y^{(\alpha-2)(1+l')}M_k^S(y^{-1}, \theta y^{\alpha})B_{j,j-l'}(y^{-1}, \theta y^{\alpha})s_{>,l}(y)e^{-2i\phi(y)}dy+r\end{array}$$
As every term $\int_{x_0}^{+\infty} y^{(\alpha-2)(1+l')}M_k^S(y^{-1}, \theta y^{\alpha})B_{j,j-l'}(y^{-1}, \theta y^{\alpha})s_{>,l}(y)e^{-2i\phi(y)}dy$ is bounded (using integration by parts and properties of $s_{>, l'}$ for all $l'$), this ends the proof of the first equality of Proposition \ref{p:all-terms}. The second equality is a consequence of $\vert s^{>}_{j}(x)\vert\leq C_j\theta^{\frac{j}{\alpha}}$ for all $j$. Proposition \ref{p:all-terms} is proven.

\section{Properties of the Gamma function}\label{section:gamma}
We recall the expression of the Gamma function by a semi-convergent integral for $0<\delta<1$:
\be\label{semi-convergent}\int_0^{+\infty}t^{\delta-1}e^{-it}dt=i^{-\delta}\Gamma(\delta).\ee
The equality (\ref{semi-convergent}) is a consequence of 
$$\Gamma(\delta)=k^{\delta}\int_0^{+\infty}t^{\delta-1}e^{-kt}dt, \Re k>0$$
The generalization to $\alpha>0$ arbitrary is straightforward, thanks to ($\alpha$ not an integer)
\begin{equation}\label{gamma-0}\Gamma(\alpha+1)=\alpha (\alpha-1)... (\alpha-[\alpha]+1)\Gamma(\alpha-[\alpha]),\end{equation}
which ensures that
$$\Gamma(\alpha+1)=k^{\alpha+1}\int_0^{+\infty}t^{\alpha}e^{-kt}dt, \Re k>0$$
from which one deduces
\begin{equation}\label{eq:Gamma}\frac{\Gamma(\alpha+1)}{(2i)^{\alpha+1}}=\int_0^{x_0}y^{\alpha}e^{-2iy}dy+(2i)^{-n}\int_{x_0}^{+\infty}\frac{d^n}{dy^n}(y^{\alpha})e^{-2iy}dy+e^{-2ix_0}\sum_{p=0}^{n-1}\frac{\frac{d^p}{dy^p}(y^{\alpha})}{(2i)^{p+1}}(x_0).\end{equation}
This equality is a consequence of
$$\Gamma(\delta)= (2i)^{\delta}\int_0^{x_0} t^{\delta-1}e^{-2it}dt+(2i)^{\delta-1}[x_0^{\delta-1}e^{-2ix_0}+\int_{x_0}^{+\infty}(\delta-1)t^{\delta-2}e^{-2it}dt]$$
obtained by choosing $k=\epsilon+2i$, $\epsilon>0$ and to notice that
$$\Gamma(\delta)=(\epsilon +2i)^{\delta}\int_0^{x_0} t^{\delta-1}x_0^{\delta-1}e^{-(\epsilon +2i)t}dt+(\epsilon+2i)^{\delta-1}[e^{-(2i+\epsilon)x_0}+\int_{x_0}^{+\infty}(\delta-1)t^{\delta-2}e^{-(\epsilon +2i)t}dt]$$
then using $\epsilon\rightarrow 0_+$ and the dominated convergence theorem\\
Choose $\delta=\alpha-[\alpha]\in (0,1)$. One then has, thanks to (\ref{gamma-0})
$$\Gamma(\alpha+1)= (2i)^{\delta}\int_0^{x_0} \alpha.(\alpha-1)....\delta t^{\delta-1}e^{-2it}dt+(2i)^{\delta-1}[\alpha.(\alpha-1)....\delta x_0^{\delta-1}e^{-2ix_0}+\int_{x_0}^{+\infty}\alpha.(\alpha-1)....\delta(\delta-1)t^{\delta-2}e^{-2it}dt],$$
which rewrites
$$\Gamma(\alpha+1)= (2i)^{\delta}\int_0^{x_0} \frac{d^{[\alpha]+1}}{dt^{[\alpha]+1}}(t^{\alpha})e^{-2it}dt+(2i)^{\delta-1}[\alpha.(\alpha-1)....\delta x_0^{\delta-1}e^{-2ix_0}+\int_{x_0}^{+\infty}\frac{d^{[\alpha]+1}}{dt^{[\alpha]+1}}(t^{\alpha})e^{-2it}dt],$$
Using repetitively integration by parts on the first term leads to equality (\ref{eq:Gamma}).

\end{document}